\documentclass[12pt]{amsart}

\usepackage{amsmath,amscd,amssymb,amsfonts,amsthm,mathabx,hyperref}
\usepackage{xcolor}  
\hypersetup{   colorlinks,   linkcolor={red!50!black},    citecolor={blue!70!black},   urlcolor={blue!80!black}}

\usepackage{tcolorbox}
\usepackage[cmtip,matrix,arrow]{xy}
\usepackage{tikz}
\usepackage{tikz-cd}

\newtheorem{theorem}{Theorem}[section]

\newtheorem{proposition}[theorem]{Proposition}

\newtheorem{lemma}[theorem]{Lemma}
\theoremstyle{definition}    
\newtheorem{definition}[theorem]{Definition}
\theoremstyle{remark}

\newtheorem{remark}[theorem]{Remark}

\newtheorem{example}[theorem]{Example}
\newtheorem{examples}[theorem]{Examples}
\newcommand\A{\mathcal{A}}
\renewcommand{\AA}{\mathbb{A}}

\newcommand\M{\mathcal{M}}
\newcommand\G{\mathcal{G}}

\renewcommand{\L}{\mathcal{L}}
\renewcommand{\O}{\mathcal{O}}

\newcommand{\ca}{\mathcal}

\newcommand{\E}{\ca{E}}

\newcommand{\R}{\mathbb{R}}
\newcommand{\C}{\mathbb{C}}

\newcommand{\Z}{\mathbb{Z}}

\newcommand{\FF}{{\mathsf{F}}}

\newcommand\lie[1]{\mathfrak{#1}}
\renewcommand{\k}{\lie{k}}
\newcommand{\h}{\lie{h}}
\newcommand{\g}{\lie{g}}

\renewcommand{\a}{\mathsf{a}}

\newcommand{\on}{\operatorname}
\newcommand{\At}{ \on{At} }

\newcommand{\Aut}{ \on{Aut} } 
\newcommand{\Gau}{ \on{Gau} }
 \newcommand{\aut}{ \mf{aut} } 
 \newcommand{\gau}{ \mf{gau} }
\newcommand{\Ad}{ \on{Ad} }
\newcommand{\ad}{\on{ad}}

\newcommand{\SO}{ \on{SO}}

\newcommand{\vol}{  \on{vol}}

\newcommand\qu{/\kern-.7ex/} 

\newcommand{\lra}{\longrightarrow}
\newcommand{\hra}{\hookrightarrow}

\renewcommand{\d}{{\mathsf{d}}}
\newcommand{\ol}{\overline}

\newcommand{\f}{\frac}

\newcommand{\p}{\partial}
\renewcommand{\l}{\langle}
\renewcommand{\r}{\rangle}
\newcommand\hh{{\f{1}{2}}}

\newcommand{\eeq}{\end{eqnarray*}}
\newcommand{\beq}{\begin{eqnarray*}}

\newcommand{\pr}{\on{pr}}
\newcommand{\wh}{\widehat}
\newcommand{\wt}{\widetilde}
\newcommand{\mf}{\mathfrak}
\newcommand{\rra}{\rightrightarrows}
\newcommand{\GL}{\on{GL}}

\newcommand{\Vect}{\on{Vect}}
\newcommand{\Diff}{\on{Diff}}
\newcommand{\PSL}{\on{PSL}}
\newcommand{\PGL}{\on{PGL}}
\newcommand{\SL}{\on{SL}}
\newcommand{\RP}{\R\!\on{P}}
\renewcommand{\S}{\ca{S}}
\newcommand{\vir}{\mf{vir}}

\newcommand{\CC}{\mathsf{C}}
\newcommand{\PP}{\mathsf{P}}
\newcommand{\ignore}[1]{}
\newcommand{\ez}{\mathsf{e}}
\newcommand{\sz}{\mathsf{s}}
\newcommand{\tz}{\mathsf{t}}
\newcommand{\oz}{\mathsf{o}}
\newcommand{\gz}{\mathsf{g}}
\renewcommand{\subset}{\subseteq}

\newcommand{\DD}{\mathbb{D}}
\newcommand{\HH}{\mathbb{H}}

\newcommand{\pS}{\partial{\Sigma}}

\renewcommand{\supset}{\supseteq}
\newcommand{\al}{\alpha}
\newcommand{\az}{\mathsf{a}}

\newcommand{\zz}{\phantom{\cdot}^0 }
\newcommand{\bb}{\phantom{\cdot}^b }

\begin{document}
\sloppy

\title[Teichm\"uller spaces for surfaces with boundary]{Symplectic geometry of Teichm\"uller spaces\\ for surfaces with ideal boundary}
\author{Anton Alekseev}
\author{Eckhard Meinrenken}

\begin{abstract}	
A hyperbolic 0-metric on a surface with boundary is a hyperbolic metric on its interior, exhibiting the boundary behavior of the standard metric on the Poincar\'e disk.  Consider the infinite-dimensional Teichm\"uller spaces of hyperbolic 0-metrics on oriented surfaces with boundary, up to  diffeomorphisms fixing the boundary and homotopic to the identity. We show that these spaces have natural  symplectic structures,
depending only on the choice of an invariant metric on $\mf{sl}(2,\R)$. We prove that these Teichm\"uller spaces 
are Hamiltonian Virasoro spaces for the action of the universal cover of the group of diffeomorphisms of the boundary.  
We give an explicit formula for the Hill potential on the boundary defining the moment map. Furthermore, using Fenchel-Nielsen 
parameters we prove a Wolpert formula for the symplectic form, leading to global Darboux 
coordinates on the Teichm\"uller space. 
 \end{abstract}

\maketitle
\tableofcontents

\section{Introduction}
A hyperbolic structure on a compact, oriented surface
$\Sigma$ without boundary may be described by an atlas with oriented charts taking values in the Poincar\'{e} disk $\DD$, with constant transition functions given by orientation preserving isometries of $\DD$. The same definition may be used for surfaces $\Sigma$ 
\emph{with boundary}, using as the model space the \emph{closed} Poincar\'{e} disk $\ol{\DD}$. Given a hyperbolic structure, the interior of the surface acquires a 
hyperbolic metric, exhibiting the same boundary behaviour as the standard metric on the 
Poincar\'{e} disk. Metrics of this type are known as \emph{conformally compact hyperbolic metrics} or \emph{hyperbolic 0-metrics}.
The boundary components are regarded as a boundary at infinity, called \emph{ideal boundary}. One pictures $\Sigma$ as a surface with funnel ends, also known as \emph{trumpets}: 
\begin{center}
	\includegraphics[scale=0.6]{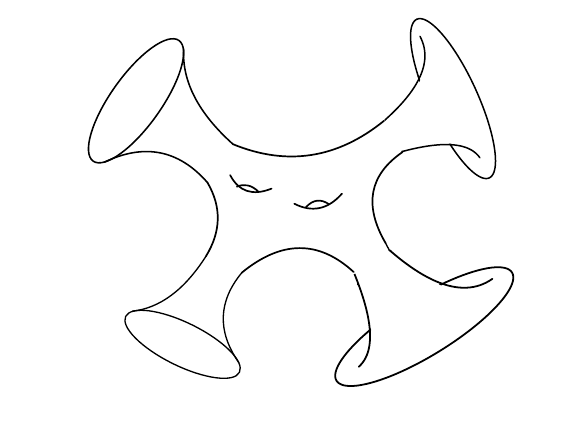}
\end{center}

In this paper, we consider the Teichm\"uller space of hyperbolic structures, 
\[ \on{Teich}(\Sigma)=\on{Hyp}(\Sigma)/\zz\Diff_\oz(\Sigma),\]
where $\zz\Diff_\oz(\Sigma)$ are the diffeomorphisms fixing the boundary and isotopic to the identity. If $\p\Sigma
\neq\emptyset$, the space $\on{Teich}(\Sigma)$ is infinite-dimensional. It has a 
residual action of the universal cover of the group $\Diff_\oz(\pS)$ of orientation preserving diffeomorphisms of the boundary, 
and of the mapping class group  $\on{MCG}(\Sigma)$. 
The infinite-dimensional Teichm\"uller space, and corresponding Riemann moduli space $\M(\Sigma)=\on{Teich}(\Sigma)/\on{MCG}(\Sigma)$, have been  
studied in the literature from the perspective of complex geometry and quasi-conformal mappings. See  for example
Bers \cite[Section 19]{ber:fin}, Thurston \cite[Remark 4.6.17]{thu:3dim}, and the work of Takhtajan-Teo \cite{tak:wei,tak:wei2}; see Schippers-Staubach \cite{schippers2023weilpetersson} for a comprehensive overview. However, it appears that the symplectic aspects of this 
space have been neglected. We will show: 
\bigskip

\noindent{\bf Theorem A.} The space $\on{Teich}(\Sigma)$ has a natural (weak) symplectic structure, invariant under the action of the mapping class group and of the 
universal cover of $\Diff_\oz(\pS)$. 
\bigskip

The action of the universal cover of $\Diff_\oz(\pS)$ admits a moment map, turning $\on{Teich}(\Sigma)$ into a Hamiltonian Virasoro space.  Recall that 
the Virasoro  Lie algebra $\mf{vir}(\p\Sigma)$ is a central extension of the Lie algebra $\Vect(\pS)$ of vector fields on the boundary. 
Its smooth dual space comes with a map to $\R$, and the affine hyperplane 
$\vir(\pS)^*_1$
at level $1$ is identified with the space $\on{Hill}(\pS)$ of Hill operators on the boundary. 
Given a local coordinate $x$ on the boundary, 
Hill operators are of the form $L=\f{d^2}{d x^2}+T$ for a Hill potential $T(x)$.
The action of diffeomorphisms on Hill potentials $T(x)$ is given by the expression 
\[ (\FF^{-1}. T)(x)=\FF'(x)^2\,T(\FF(x))+\hh \S(\FF)(x)\] 
involving the \emph{Schwarzian derivative} $\S$. 

	\bigskip\noindent{\bf Theorem B.} 
	The action of the universal cover of $\Diff_\oz(\pS)$ on $\on{Teich}(\pS)$ is Hamiltonian, with a canonically defined equivariant moment map 
	$\Phi\colon \on{Teich}(\Sigma)\to \on{Hill}(\pS)^-$. 
\bigskip

Here $\on{Hill}(\pS)^-=\vir(\pS)^*_{-1}$ is the affine subspace at level $-1$ (consisting of 
 all $-L$ where $L$ is a Hill operator).
 The quickest description of the moment map  is based on the observation that a hyperbolic structure on $\Sigma$ determines a projective structure on the boundary (by restricting the $\ol{\DD}$-valued charts of $\Sigma$ to the boundary). As is well-known (see e.g. \cite{ovs:pro}), projective structures on oriented 1-manifolds are equivalent to Hill operators.  The moment map takes $[\gz]$ to $-L$ where $L$ is the Hill operator for the projective structure. For an explicit description, choose a local coordinate 
$x$ on the boundary, and complete  to a local coordinate system $x,y$ where $y$ is a boundary defining function. 
Let $a(x)>0$ 
be the positive function obtained from the most singular part of the 
Riemannian volume form: $\d\vol_{\mathsf{g}}=y^{-2} a(x)\d x\wedge \d y+O(y^{-1})$.  Let $k(x,y)$ be the geodesic curvature
of the curve $(x,y)\mapsto (x+t,y)$; one finds $k(x,y)=1+c(x)y^2+O(y^3)$.  

\bigskip\noindent{\bf Theorem C.} 
The boundary Hill potential for the hyperbolic 0-metric $\gz$ is given by 
\[ T=\hh \left(\f{a''}{a}-\f{3}{2}\Big(\f{a'}{a}\Big)^2\right) +\f{a^2}{2} c.\]
 \smallskip
 
If $\Sigma$ has negative Euler characteristic, we may use a pants decomposition 
to obtain a Fenchel-Nielsen parametrization  
\[ \on{Teich}(\Sigma)=(\R_{>0}\times \R)^{3g-3+r}\times (\R_{>0}\times \wt{\Diff}_\oz(S^1))^r.
\]
Let $(\ell_i,\tau_i)$ be the length and twist parameters for the first $3g-3+r$ factors (corresponding to gluing circles), and $\ell_j,\FF_j$ the parameters for the last $r$ factors (corresponding to trumpet ends).  We have the following version of Wolpert's formula for the symplectic form (Theorems \ref{th:fn} and Theorem \ref{th:trumpet}):

\bigskip\noindent{\bf Theorem D.} 
In Fenchel-Nielsen parameters, the symplectic form on $\on{Teich}(\Sigma)$ is given by 
 \[ \omega=	\hh \sum_{i=1}^{3g-3+r} \d \ell_i\wedge \d \tau_i
 -\f{1}{4}\sum_{j=1}^r\d \int_{S^1}\big(\ell_j^2\FF'_j\,\d\FF_j +(\FF'_j)^{-1}\d\FF''_j
 \big)\]
 Here $\d$ is the exterior differential on the diffeomorphism group, and $'$ denotes a derivative in the $S^1$-direction.  
\bigskip

The terms in the second sum are symplectic forms on the factors $\R_{>0}\times \wt{\Diff}_\oz(S^1)$, which may be interpreted as moduli spaces for the trumpet (with one geodesic end and one ideal boundary). We shall show (Proposition \ref{prop:darboux}) how to introduce global Darboux coordinates for the trumpet moduli space,  resulting in global Darboux coordinates for the space $\on{Teich}(\Sigma)$. 

One of the motivations for this work are  recent developments in the physics literature on Jackiw-Teitelboim gravity, notably the articles by Saad-Shenker-Stanford \cite{saa:jt}, Maldacena-Stanford-Yang \cite{mal:con}, Cotler-Jensen-Maloney \cite{cot:low}, 
and Stanford-Witten \cite{sta:jt}. The discussion in these articles involves hyperbolic surfaces with funnel ends (`trumpets'), 
using cut-off along `wiggly boundaries'  to create surfaces of  finite volume, leading to a theory governed by a \emph{Schwarzian action}. As shown by the physicists, this relates JT gravity to mathematical concepts such as the  Mirzakhani  recursion formulas for Weil-Petersson volumes, topological recursion and random matrix theory, Duistermaat-Heckman theory for Virasoro coadjoint orbits, and more.  

A second motivation is our program to develop a theory of Hamiltonian Virasoro spaces, analogous to the theory of Hamiltonian loop group spaces \cite{me:lo}. An important example of  such a space is the infinite dimensional moduli space 
\[ \M_G(\Sigma)=\frac{ \{ A \in \Omega^1(\Sigma, \mathfrak{g})|\   \d A+\hh[A,A]=0\}}{ \{g\colon \Sigma\to G \colon g|_{\partial \Sigma} =e\} }\]
of flat $G$-connections, where $G$ is a simply connected Lie group with an invariant metric on its Lie algebra. This space has a  symplectic form given by reduction, and the residual 
action of $\on{Map}(\p\Sigma,G)$ is Hamiltonian, with affine moment map given by the pullback of the connection, $[A]\mapsto \iota_{\pS}^*A$. It is natural to have a similar example for the Virasoro setting; in fact the two situations are related by Drinfeld-Sokolov reduction. 
As shown in \cite{al:mom}, Hamiltonian loop group spaces with proper moment map are equivalent to finite-dimensional quasi-Hamiltonian $G$-spaces; the results of  \cite{al:coad} pave the way for a similar correspondence for Hamiltonian Virasoro spaces.

Let us briefly summarize our construction of the symplectic form on $\on{Teich}(\Sigma)$.
By a famous result of Goldman \cite{gol:top} and Hitchin \cite{hit:self}, the  
Teichm\"uller space for a surface $\Sigma$ \emph{without boundary}, of genus $\gz\ge 2$, 
is a moduli space $\A_{\on{flat}}(P)/\Gau(P)$ of flat connections on a principal $G$-bundle $P\to \Sigma$ of Euler number $2\mathsf{g}-2$. In particular, the Weil-Petersson symplectic form 
is obtained by reduction of the Atiyah-Bott symplectic structure on the space of connections. There does not seem to be an immediate generalization of this result to the case 
of non-empty boundary. Instead, motivated by the theory of geometric structures \cite{gol:geo}
we take as our starting point is a principal $G$-bundle $P\to \Sigma$ together with a $G$-equivariant morphism 
$\sigma\colon P\to \ol{\DD}$, called a \emph{developing section}. For suitable choice of $(P,\sigma)$, we define 
\[ \wh{\on{Teich}}(\Sigma)=\A_{\on{flat}}^{\on{pos}}(P)/\Aut_\oz(P,\p P,\sigma),\]
the quotient of the space of flat connections satisfying a certain positivity condition with respect to $\sigma$, by the identity component of automorphisms preserving $\sigma$ and 
trivial along the 
boundary. (This space may be interpreted as elements of $\on{Teich}(\Sigma)$ together with 
developing sections for the projective structure on the boundary.) We show that this space is a symplectic quotient for the Atiyah-Bott symplectic structure on $\A^{\on{pos}}(P)$. 
It comes with a residual action of $\Gau(\p P,\p\sigma)$, the gauge transformations of $\p P=P|_{\pS}$ preserving $\p\sigma=\sigma|_{\p P}$. We prove that the latter action  
is Hamiltonian, with moment map image a single coadjoint orbit $\O$, and 
\[ \on{Teich}(\Sigma)=\big(\wh{\on{Teich}}(\Sigma)\times \O^-
\big)\qu \Gau(\p P,\p\sigma)\]
(a symplectic reduction). This defines the symplectic structure on $\on{Teich}(\Sigma)$ (Theorem A). The moment map (Theorem B) is obtained by explicit calculation, beginning with the moment map for the action of the full group of automorphisms on the space $\A(P)$. The relevant background material on the Atiyah-Bott construction is provided in the appendix.

The structure of the paper is as follows. In Section 2, we recall basic definitions and properties of hyperbolic structures on surfaces with boundary,
and introduce infinite dimensional Teichm\"uller spaces ${\rm Teich}(\Sigma)$. In Section 3, using the moving coframe formalism of \'{E}. Cartan, we describe the relation between hyperbolic metrics and flat $\mf{sl}(2, \mathbb{R})$ connection 1-forms. We prove local normal forms for coframes; as a by-product this gives a new proof of the local normal form for hyperbolic 0-metrics near the ideal boundary. In Section 4, we use a more global approach, considering  principal ${\rm PSL}(2, \mathbb{R})$ bundles $P\to \Sigma$ equipped with a developing section $\sigma$. For suitable choice of $P$, we exhibit $\on{Teich}(\Sigma)$ as a space of $\sigma$-positive flat connections modulo a subgroup of bundle automorphisms preserving $\sigma$. Section 5 is devoted to a detailed study of this group of automorphisms. In Section 6, we apply the Atiyah-Bott construction to the space of positive connections, and explain how to obtain  ${\rm Teich}(\Sigma)$ by reduction.
In Section 7, we prove that the diffeomorphisms of the boundary act on ${\rm Teich}(\Sigma)$, that this action is Hamiltonian, and that it corresponds to the Virasoro central extension of the diffeomorphism group. In Section 8, we give an explicit description of the symplectic structure in Fenchel-Nielsen parameters, and give a construction of Darboux coordinates on ${\rm Teich}(\Sigma)$ which combines the classical Wolpert formula \cite{wol:sym} with the construction of Darboux coordinates on hyperbolic Virasoro coadjoint orbits \cite{al:bos}. 

\bigskip

{\bf Acknowledgments.} We are grateful to S.~Ballas, D.~Borthwick, O.~ Chekeres, W.~ Goldman, P.~ Hekmati, N.~Higson, Y.~ Loizides, J.-M.~ Schlenker, S.~ Shatashvili, J.~ Sonner, T.~ Strobl, and D.~ Youmans for useful discussions.
Research of A.A.~ was supported in part by the grants 208235 and 200400 and by the National Center for Competence in Research (NCCR) SwissMAP of the Swiss National Science Foundation,
and by the award of the Simons Foundation to the Hamilton Mathematics Institute of the Trinity College Dublin under the program “Targeted Grants to Institutes”. E.M.~ thanks P.~Hekmati for the opportunity to present this work in lectures at the University of Auckland; his research was supported by Discovery Grant RGPIN-2022-05254 from NSERC.

\bigskip

\section{Hyperbolic structures on surfaces with boundary}

\subsection{Hyperbolic and projective structures}
The model space for hyperbolic structures on surfaces $\Sigma$ without boundary is the Poincar\'{e} disk 
\[ \DD=\{z\in \C|\ |z|<1\},\] 
with the action of $G=\on{PSU}(1,1)$ by M\"obius transformations. The $G$-action on $\DD$ extends to the closed Poincar\'{e} disk
$\ol{\DD}$; this will be our model space for surfaces with boundary: 
\begin{definition}\label{def:hyperbolic}
	A \emph{hyperbolic structure} on an oriented surface $\Sigma$ with boundary  $\p\Sigma$ is an oriented atlas with  $\ol{\DD}$-valued charts, with constant transition maps given by elements of $G$. The space of all hyperbolic structures on $\Sigma$  will be denoted 
	$\on{Hyp}(\Sigma)$. 
\end{definition}
%

\begin{remark}\label{rem:schottky}
A hyperbolic structure on $\Sigma$ pulls back to a hyperbolic structure on every covering space of $\Sigma$. 
If $\Sigma$ is compact and connected, then the universal covering space is of the form $\ol{\DD}-\mf{L}$ where $\mf{L}\subset \p\DD$ is a set of \emph{limit points}, and 
	\begin{equation}\label{eq:schottky}
	\Sigma=(\ol{\DD}-\mathfrak{L})/\Gamma
	\end{equation}
	where $\Gamma\subset G$ is a Fuchsian group of Schottky type. See \cite{bor:spec} or \cite{ser:lec}.  
\end{remark}

In a similar way, taking the boundary of the Poincar\'{e} disk as the model space for projective structures, we define: 
\begin{definition}\label{def:projective}
	A \emph{projective structure} on an oriented 1-manifold $\CC$ (without boundary) is an oriented atlas with $\partial\DD$-valued charts, with constant transition maps given by elements of $G$. The space of all projective structures on $\CC$  will be denoted 
	$\on{Proj}(\CC)$. 
\end{definition}
A hyperbolic structure on $\Sigma$ determines a projective structure on the boundary -- one simply restricts the $\ol{\DD}$-valued charts.  This gives a canonical map 
\begin{equation}\label{eq:hypproj}
\on{Hyp}(\Sigma)\to \on{Proj}(\pS).\end{equation}

\subsection{Hyperbolic 0-metrics}
The $G=\on{PSU}(1,1)$-action on the Poincar\'{e} disk preserves the Poincar\'{e} metric, written in polar coordinates $z=r e^{i\varphi}$ as
\begin{equation}\label{eq:poincaredisk}
\f{4}{(1-r^2)^2}(\d r^2+r^2\d\varphi^2). 
\end{equation}
Hence, a hyperbolic structure on $\Sigma$ determines a Riemannian metric $\gz$ on the interior $\on{int}(\Sigma)$, by pulling back \eqref{eq:poincaredisk} under the coordinate charts.  At the boundary $\p\Sigma$, this metric becomes singular in the same way as the metric on $\ol{\DD}$: letting $\varrho$ be any boundary defining function, the product $\varrho^2 \gz$ extends to an ordinary metric on $\Sigma$.
Riemannian metrics with this property may be seen as ordinary Euclidean metrics on the \emph{0-tangent bundle} of Mazzeo-Melrose \cite{maz:the,maz:mer}, i.e.,  the Lie algebroid 
\[ \zz T\Sigma\to \Sigma\]
whose sections are the vector fields on $\Sigma$ that vanish along the boundary.  A given 
 \emph{0-metric} is called hyperbolic if its restriction to the interior is 
hyperbolic in the sense that it has Gauss curvature  $K_{\gz}=-1$. Theorem \ref{th:normalform} below says that all hyperbolic 0-metrics 
arise from hyperbolic structures as in Definition \ref{def:hyperbolic}.

\subsection{Teichm\"uller spaces}
For groups of diffeomorphisms of a compact, oriented manifold, we use a subscript $+$ to indicate 
diffeomorphisms preserving orientation, and subscript $\oz$ to indicate diffeomorphisms isotopic to the identity. 
We recall that the universal cover of $\Diff_+(S^1)=\Diff_\oz(S^1)$ is identified with $\Z$-equivariant diffeomorphisms of $\R$. 

Let $\Sigma$ be compact, connected, and oriented. Denote by  $\bb\Diff(\Sigma)$ the diffeomorphisms preserving the boundary, and by $\zz\Diff(\Sigma)$ the subgroup of diffeomorphisms fixing the boundary pointwise.  We hence have subgroups
\begin{equation}\label{eq:bdif} \bb\Diff(\Sigma)\supset \!\bb\Diff_+(\Sigma)\supset \!\bb\Diff_\oz(\Sigma)\end{equation}
and similarly for $ \zz\Diff(\Sigma)$.
(The superscripts $0,b$ are omitted if $\pS=\emptyset$.) The mapping class group is the quotient 
$\on{MCG}(\Sigma)=\zz\Diff_+(\Sigma)/\zz\Diff_\oz(\Sigma)$.

\begin{definition}
	The (infinite-dimensional) \emph{Teichm\"uller space} is the space of hyperbolic structures on $\Sigma$, up to diffeomorphisms 
	fixing the boundary and homotopic to the identity:
	\begin{equation}\label{eq:teichmueller}
	\on{Teich}(\Sigma)=\on{Hyp}(\Sigma)/\zz\Diff_\oz(\Sigma).\end{equation}
	The  (infinite-dimensional) (Riemann) \emph{moduli space} is the quotient
	\begin{equation}\label{eq:modulispace}
	\M(\Sigma)=\on{Hyp}(\Sigma)/\zz\Diff_+(\Sigma).\end{equation}
\end{definition}

\medskip
Equivalently, $\M(\Sigma)$ is the quotient of $\on{Teich}(\Sigma)$ under the action of the mapping class group.
Both \eqref{eq:teichmueller} and \eqref{eq:modulispace} are endowed with residual actions of boundary diffeomorphisms. Since 
$\bb\Diff_\oz(\Sigma)/\zz\Diff_\oz(\Sigma)=\Diff_\oz(\pS)$, there is an induced action 
\[ \Diff_\oz(\pS)\circlearrowright \M(\Sigma).\]
Similarly, the quotient $\bb\Diff_\oz(\Sigma)/\zz\Diff_\oz(\Sigma)$ acts on $\on{Teich}(\Sigma)$. 
This group is a covering of $\Diff_\oz(\pS)$ (not always the universal cover; see examples below). 
In any case, there is an induced action 
\[ \wt{\Diff}_\oz(\pS)\circlearrowright \on{Teich}(\Sigma)\]
of the universal cover of the identity component $\Diff_\oz(\pS)$. (If $\pS$ has several components, we may also consider diffeomorphisms 
interchanging boundary components.) 
The map \eqref{eq:hypproj} descends to  maps 
\begin{equation}\label{eq:moment1} \on{Teich}(\Sigma)\to \on{Proj}(\pS),\ \ 
\M(\Sigma) \to \on{Proj}(\pS),\end{equation}
which are equivariant for these actions. As we shall explain in this paper, these will be identified as moment maps for Hamiltonian Virasoro  spaces. 


\subsection{Special cases}
One has more concrete descriptions of the Teichm\"uller spaces, as follows. Let $g$ be the genus of $\Sigma$ and 
 $r$ the number of boundary components. The Euler characteristic is thus 
$\chi(\Sigma)=2-2g-r$. 

\subsubsection{Poincar\'{e} disk} ($g=0,\ r=1$.) 
The standard metric on the closed Poincar\'{e} disk $\ol{\DD}$ 
is the unique hyperbolic 0-metric on the disk, up to diffeomorphism. The stabilizer of the standard metric under the action of $\bb\Diff_+(\ol\DD)$ is $G=\on{PSU}(1,1)$. It follows that 
$\on{Hyp}(\ol{\DD})=\bb\Diff_+(\ol{\DD})/G$, and hence
\[ \on{Teich}(\ol{\DD})=\M(\ol{\DD})=\Diff_\oz(\partial\DD)/G\]
(since every diffeomorphism in $\zz\Diff_+(\ol\DD)$ is isotopic to the identity).
This space is closely related to Bers' \emph{universal Teichm\"uller space} \cite{ber:aut,tak:wei}; it has an interpretation as  a coadjoint orbit of the Virasoro group.

\subsubsection{Hyperbolic cylinders}\label{subsubsec:cylinder} ($g=0,\ r=2$.) 
Let $\AA=S^1\times  (-\infty,\infty)$ be the infinite cylinder (open annulus), 
with coordinates $(x,u)$ where $x\in S^1=\R/\Z$. Denote by $\ol{\AA}=S^1\times [-\infty,\infty]$ its 
compactification, with 
boundary defining function $\varrho(x,u)=\cosh(u)^{-1}$. Here $\on{MCG}(\ol{\AA})=\Z$, generated by Dehn twists
$(x,u)\mapsto (x+f(u),u)$ for $f\in C^\infty(\R)$ with $f(u)=0$ for $u<-R$ and $f(u)=1$ for $u>R$, for some $R>0$. 

Given $\ell>0$, the formula 
\begin{equation}\label{eq:doubletrumpet}
\gz= \cosh^2(u) \ell^2 \d x^2 +\d u^2
\end{equation}
defines a hyperbolic 0-metric, with the curve $u=0$  as its  unique closed simple geodesic. 
\begin{center}
	\includegraphics[scale=0.5]{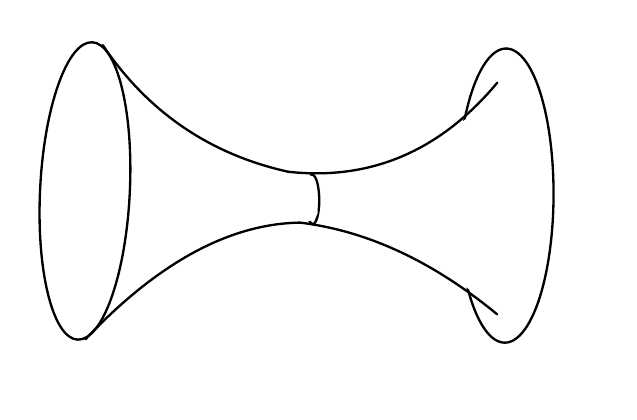}
\end{center}

Conversely, every hyperbolic 0-metric on $\ol{\AA}$ admits a unique 
closed simple geodesic; letting $\ell$ be its length, the metric is obtained from \eqref{eq:doubletrumpet} by the action of an 
element of $\Diff_+(\ol{\AA})$ preserving the two boundary components.
The stabilizer of \eqref{eq:doubletrumpet} under this action is $S^1$, acting by rotations of the cylinder. 
This gives 
\[ \M(\ol{\AA})=
\R_{>0}\times (\Diff_\oz(S^1)\times \Diff_\oz(S^1))/S^1\]
with the anti-diagonal embedding of $S^1$. For the Teichm\"uller space, we note that compactly supported diffeomorphisms of 
$(-\infty,0]\times S^1$, modulo the subgroup fixing the boundary, is the universal cover 
$\wt{\Diff}_\oz(S^1)$. This gives 
\[  \on{Teich}(\ol{\AA})=\R_{>0}\times (  \wt{\Diff}_\oz(S^1)\times \wt{\Diff}_\oz(S^1)    )/\R\]
where  $\R$ is embedded 
anti-diagonally.  
For $\ol{\AA}$ with the 0-metric \eqref{eq:doubletrumpet}, the subset given by $u\ge 0$ is called a \emph{trumpet} (also called \emph{flare} or  \emph{funnel}). 
It has one geodesic boundary component and one boundary at infinity. 

\subsubsection{Surfaces of negative Euler characteristic ($g=0,\ r>2$ or $g\ge 1,\ r\ge 1$)}\label{subsec:negeuler}
After choice of a Fenchel-Nielsen parametrization, one finds 
\begin{equation}\label{eq:FN}
  \on{Teich}(\Sigma)\cong (\R_{>0}\times \R)^{3g-3+r}\times \prod_{j=1}^r (\R_{>0}\times \wt{\Diff}_\oz(S^1)).\end{equation}
We will give details in Section \ref{sec:fn} below. At this point, we just mention that the Fenchel-Nielsen parametrization involves a pants decomposition of the surface. The boundaries of the pants which are not among the boundaries of $\Sigma$ 
form a system of 
$3g-3+r$ circles separating the pants; given a hyperbolic 0-metric these are realized as geodesics, and 
the factors $\R_{>0}\times \R$ in \eqref{eq:FN} are the corresponding length and twist parameters. The boundary components of pants 
which are also boundaries of the surface correspond to trumpets as discussed in \ref{subsubsec:cylinder}; the 
$\R_{>0}\times \wt{\Diff}_\oz(S^1)$-factors signify the length of the geodesic end of the trumpet and a twist by rotating the ideal boundary of the trumpet.


\section{Coframe formalism}

For calculations in this section, we will prefer the half-plane model of hyperbolic geometry. Let
\[ \HH=\{z=x+iy| y>0\},\]
with the standard hyperbolic metric 
\begin{equation}\label{eq:poincaremetric}
\f{1}{y^2}(\d x^2+\d y^2).
\end{equation}
The map $\HH\to \DD,\ z\mapsto (z-i)/(1-iz)$ is an  isometric isomorphism, equivariant with respect to  
\begin{equation}\label{eq:psu11}
 \Ad_q\colon\PSL(2,\R)\to \on{PSU}(1,1),\ \ \ q= \Big[\begin{array}{cc}  1& -i\\ -i& 1\end{array}\Big].\end{equation}
Here the bracket notation $[A]\in \PGL(2,\C)$ denotes the image of an element $A\in \GL(2,\C)$. The isomorphism extends to a bijection 
$\ol{\HH}\cup\{\infty\}\to \ol{\DD}$ taking $0\in \p\HH$ to $-i\in \p\DD$. 
Throughout this section, $G$ denotes the group $\PSL(2,\R)$, and $\g$ its Lie algebra $\mf{sl}(2,\R)$. 

\subsection{Cartan coframe formalism}
A convenient tool for dealing with 0-metrics is the \emph{moving coframe formalism} due to \'{E}. Cartan. Let $\Sigma$ be an oriented surface with boundary. Let  $\zz\Omega^k(\Sigma)=\Gamma(\bigwedge^k\,(\zz T\Sigma)^*)$ denote the de Rham complex
of the Lie algebroid $\zz T\Sigma$. Elements of this space may be seen as ordinary $k$-forms on $\on{int}(\Sigma)$ such that $\varrho^k\alpha$ extends to all of $\Sigma$, for any boundary defining function $\varrho$.

 An \emph{oriented coframe} over an open subset $U\subset \Sigma$ is 
pair of 0-covector fields \[\alpha_1,\alpha_2\in\zz\Omega^1(U)\]
giving an oriented frame for $(\zz T \Sigma)^*|_U$. 	
Given a 0-metric $\mathsf{g}$ on $\Sigma$, the coframe is \emph{orthonormal} if 
\[ \mathsf{g}|_U=(\alpha_1)^2+(\alpha_2)^2.\]
In this case, the Riemannian volume form $\d\vol_{\mathsf{g}} \in \zz\Omega^2(\Sigma)$ is given by 
\[ \d\vol_{\mathsf{g}}|_U =\alpha_1\wedge\alpha_2.\] 
Any two oriented orthonormal coframes for $\gz$ are related,   
on the overlap of their domains, by a \emph{coframe rotation}  
\begin{equation}\label{eq:framerotation}  \al_1'=\cos\varphi  \al_1+\sin\varphi  \al_2,\ \ \al_2'=-\sin\varphi  \al_1+\cos\varphi  \al_2
\end{equation}
with $\varphi\in C^\infty(U\cap U',\R)$. The  \emph{spin connection} for an oriented coframe is the $0$-covector field
\[ \kappa\in \zz\Omega^1(U)\]
defined by the equations 
$\d\alpha_1=-\kappa\wedge \alpha_2,\ \ \d\alpha_2=\kappa\wedge\alpha_1.$
%
Under coframe rotation, the spin connection changes to $\kappa'=\kappa-\d\varphi$. 
This shows that there exists a globally defined function $K_{\mathsf{g}}\in C^\infty(\Sigma)$ such that 
\[ K_{\mathsf{g}}\, \d\vol_{\mathsf{g}}\big|_U=\d\kappa.\]
This function is  the \emph{Gauss curvature} of the metric. 
The three equations 
\begin{equation}\label{eq:structureequations}
\d\alpha_1=-\kappa\wedge\alpha_2,\ \ \d\alpha_2=\kappa\wedge\alpha_1,\ \ \d\kappa=K_{\mathsf{g}}\ \alpha_1\wedge\alpha_2\end{equation}
are  \emph{Cartan's structure equations}.

\begin{examples}\label{ex:example2}
	We list some standard coframes for hyperbolic 0-metrics, and the 
	resulting spin connections. 
	\begin{enumerate}
		\item\label{it:a} Upper half plane $\bar\HH$: 
		\[ \alpha_1=\f{\d x}{y},\ \alpha_2=\f{\d y}{y},\ \kappa=-\f{\d x}{y}.\] 
		\item\label{it:b}  Poincare disk $\ol{\DD}$: 
		\[ \alpha_1=\f{2}{1-r^2} r\d\varphi,\ \ \alpha_2=\f{2}{1-r^2} \d r,\ \ 
		\kappa=\f{1+r^2}{1-r^2}\d\varphi.\]
		Replacing  $r$ with the boundary defining function  
		$y=\f{1-r}{1+r}$, this becomes 
		\[ \alpha_1=\f{1-y^2}{2} \f{\d\varphi}{y},\ \ 
		\alpha_2=\f{\d y}{y},\ \ \kappa=-\f{1+y^2}{2} \f{\d\varphi}{y}\]
		which more clearly exhibits the coframe as 0-covector fields. 
		\item\label{it:c}  Hyperbolic cylinder $\ol{\AA}=S^1\times [\infty,\infty]$ 
		with parameter $\ell>0$: Using coordinates $x,u$,  
		\[ \alpha_1=\cosh(u)\ell\ \d x ,\ \   \alpha_2=-\d u,\ \ 
		\kappa=-\sinh(u)\ell\d x.\]
		Putting $y=e^{-u}$ (which is a boundary defining function near the boundary  $u=\infty$), this becomes
		\[ \alpha_1=\ell\f{1+y^2}{2}\ \f{\d x}{y},\ \  
		\alpha_2=\f{\d y}{y},\ \ 
		\kappa=-\ell\f{1-y^2}{2}\ \f{\d x}{y}.
		\]
		\item\label{it:d}  Fefferman-Graham coframe: For any function $T(x)$, the formulas 
		\[ \alpha_1=\big(1-y^2 T(x)\big) \f{\d x}{y},\ \  \alpha_2=\f{\d y}{y},\ \ 
		\kappa=-\big(1+y^2 T(x)\big) \f{\d x}{y}
		\]
		define a coframe for a hyperbolic 0-metric on $\{(x,y)|\ y\ge 0,\ y^2 T(x)<1\}$. This unifies the boundary behaviour of the previous examples. 
	\end{enumerate}
\end{examples}	
\begin{remark}\label{rem:adaptedcoord}
From now on, the letters $x,y$ will be reserved for local coordinates on $U\subset \Sigma$ that are \emph{oriented} 
(i.e., $\d x\wedge \d y>0$) and \emph{adapted to the boundary}, in the sense that 
$y$ is a boundary defining function for $\partial\Sigma\cap U$. If $U$ is contained in the interior, this just means $y>0$ everywhere on $U$.
\end{remark}

\subsection{The connection 1-form associated with an orthonormal coframe}\label{subsec:flat}
Given a local orthonormal coframe $\alpha_1,\alpha_2\in \zz\Omega^1(U)$ for a $0$-metric $\gz$, with associated spin connection $\kappa$, define a 0-connection 1-form
\begin{equation}\label{eq:A}
A=\hh \left(\begin{array}{cc}  \al_2 & \al_1-\kappa \\   \al_1+\kappa& -\al_2\end{array}\right)
\in \zz\Omega^1(U,\g)
.\end{equation}

The  curvature $F_A=\d A+\hh [A,A]$  is given by 
\begin{equation}\label{eq:curvature} 
F_A= (K_\gz+1)
\left(\begin{array}{cc} 0 & -1\\ 1 & 0
\end{array}
\right)
\
\d\vol_{\mathsf{g}}.\end{equation}
In particular, the 0-metric $\mathsf{g}$ is hyperbolic if and only if the connection 1-form 
$A$ is flat. Let $K\subset G=\PSL(2,\R)$ be the maximal compact  subgroup given as the stabilizer of $i\in \HH$. It is identified with $\SO(2)$ 
by the map 
\begin{equation}\label{eq:identification}
\SO(2)\to K,\ \ R(\varphi)=\left(\begin{array}{cc}  \cos(\varphi) &  -\sin(\varphi)\\  \sin(\varphi)& \cos(\varphi)\end{array}\right)\mapsto [R(\varphi/2)].
\end{equation}
Under coframe rotations  \eqref{eq:framerotation}, the  0-connection 1-form $A$ transforms by 
\begin{equation}\label{eq:aprime} A'=[R\left(\varphi/2\right)]\bullet A\end{equation}
where 
$\bullet$ signifies a gauge transformation 
\[ g\bullet A=
\Ad_g(A)-(\d g) g^{-1}\] for $g\in C^\infty(U,G)$.
We hence see that the $K$-action on connection 1-forms translates into the $\SO(2)$-action on coframes. Observe that $A$ does \emph{not} transform as a connection on the tangent bundle. 

\begin{example}\label{ex:standardex}
	For the upper half plane $\bar{\HH}$ with its standard coframe (Example \ref{ex:example2}\ref{it:a}),
	\[ A=  \left(\begin{array}{cc}  \f{1}{2y}\ \d y &  \f{1}{y}\d x\\ 0 & -\f{1}{2y}\ \d y\end{array}\right).\]
	Note that $A= g^{-1}\bullet 0$ where 
	\[ g=\left(\begin{array}{cc}  1 &  x \\  0 &  1 \end{array}\right) 	\left(\begin{array}{cc}  y^{\f{1}{2}} &  0 \\  0 &  y^{-\f{1}{2}} \end{array}\right).
	\]
\end{example}

\begin{example}[Fefferman-Graham gauge] The 0-connection form defined by the Fefferman-Graham coframe 
(Example \ref{ex:example2}\ref{it:d}) reads as 
\[ A= \left(\begin{array}{cc}  \f{1}{2y}\ \d y &  \f{1}{y}   \d x\\  -T(x)y\d x& -\f{1}{2y}\ \d y\end{array}\right).\]
\end{example}

\subsection{Adapted coframes}\label{subsec:adapted}
For any manifold $M$ with boundary $\p M$, the restriction of the 0-cotangent bundle  to the boundary has a distinguished trivial subbundle 
\[ \p M\times \R\subset   (\zz TM)^*|_{\p M},\]
spanned by the restriction of $\varrho^{-1}\d\varrho\in \zz\Omega^1(M)$, for any boundary defining function $\varrho$. (Changing $\varrho$ by a positive function changes this expression by an \emph{ordinary} (exact) 1-form, but the restriction of a regular 1-form as a section of the 0-tangent bundle vanishes.)

\begin{definition}\label{def:adapted}
	An oriented coframe $\alpha_1,\alpha_2\in \zz\Omega^1(U)$, defined on a neighborhood $U\subset \Sigma$ of the boundary, is \emph{adapted to the boundary} if 	$\alpha_2|_{\pS}$  is the canonical 
	section of $(\zz T\Sigma)^*|_{\pS}$. 
		That is, 
		\[ \alpha_2=\f{d \varrho}{\varrho}+O(\varrho^0).\]
\end{definition}
Here we  write $\alpha=O(\varrho^k)$ if $\varrho^{-k}\alpha$ extends smoothly to the boundary (as an ordinary differential form). 
Note that for an adapted coframe, $\varrho\alpha_1$ pulls back to a volume form on $\p\Sigma$.

The coframes in parts \eqref{it:a},\ \eqref{it:b},\ \eqref{it:d} of Examples \ref{ex:example2} are adapted to the boundary. 
(But \eqref{it:b} is not defined at the center of $\ol{\DD}$.) 
The coframe in \eqref{it:c} for the double trumpet is adapted to the boundary at $u=\infty$ but not at $u=-\infty$. One can turn it into an adapted coframe by applying a coframe rotation $R(\phi$ with $\phi|_{u=-\infty}=\pi,\ \phi|_{u=0}=0$. 
If $\Sigma$ is compact and connected, of non-zero Euler characteristic, 
it is impossible, by Poincar\'{e}'s theorem on zeroes of vector fields on surfaces, to find a \emph{global} oriented  coframe that is adapted to all the boundary components. On the other hand, we have: 

\begin{lemma}\label{lem:adapted}
	Every hyperbolic 0-metric $\mathsf{g}$ admits an oriented orthonormal coframe $\alpha_1,\alpha_2$ on some collar neighborhood of the boundary, which is adapted to the boundary. The spin connection of such a coframe satisfies 
	\[ \kappa=-\alpha_1+O(\varrho^1).\]
\end{lemma}

\begin{proof}
Choose an oriented orthonormal coframe $\alpha_1,\alpha_2\in \zz\Omega(U)$ on a collar neighborhood $U$ of the boundary, 
with the property that 
$\alpha_2|_{\p\Sigma}$ is a positive function times the canonical section of $(\zz T\Sigma)^*|_{\pS}$. 
Using that $\gz$ is hyperbolic, we shall show that $\alpha_2|_{\p\Sigma}$ must then be \emph{equal to} the canonical section. 
It suffices to prove this in  local coordinates $x,y$ adapted to the boundary (cf.~ Remark \ref{rem:adaptedcoord}). 
	Consider the Laurent expansions
	in powers of $y$, 
	\[ \alpha_1=\f{1}{y}\beta_1+\gamma_1+O(y^1),\ \ \alpha_2=\f{h}{y}\d y+\gamma_2+O(y^1),\ \  \kappa=\f{1}{y}\beta_3+\gamma_3+O(y^1),\]	
	where  $\beta_i,\gamma_i$,  are of the form $f_i(x)\d x+g_i(x)\ \d y$, and where $h$ is a function of $x$, with 
	$h(x)>0$. 
	Comparing coefficients of $y^{-2}$ in the structure equations \eqref{eq:structureequations} (with $K_{\gz}=-1$) gives conditions 
	\begin{equation}\label{eq:-2}
	\beta_1\wedge \d y=-h\beta_3\wedge \d y,\ \ 0=\beta_3\wedge\beta_1,\ \ \beta_3\wedge\d y=-h\beta_1\wedge \d y.
	\end{equation}
	From 
	\[ 	\d\vol_\gz=\alpha_1\wedge\alpha_2=\f{1}{y^2}h \beta_1\wedge \d y+O(y^{-1})\] 
	we see that $\beta_1,\d y$ are pointwise linearly independent; 
	in particular $\beta_1$ is non-vanishing. Hence, the second equation in \eqref{eq:-2} shows that 
	$\beta_3$ is a scalar multiple of $\beta_1$, and so the other two equations give $\beta_1=-h\beta_3,\ \beta_3=-h\beta_1$. 
	Hence $h=1$ and $\beta_3=-\beta_1$. 
	At this stage, the expressions for the coframe have simplified to 
	\[ \alpha_1=\f{1}{y}\beta+\gamma_1+O(y^1),\ \ \alpha_2=\f{1}{y}\d y+\gamma_2+O(y^1),\ \  \kappa=-\f{1}{y}\beta+\gamma_3+O(y^1)\]	
	(where we write $\beta=\beta_1$). From the sum of the first and third structure equations, 
	$\d(\alpha_1+\kappa)=-(\alpha_1+\kappa)\wedge\alpha_2$ we obtain, by comparing coefficients of $y^{-1}$, that 
	$(\gamma_1+\gamma_3)\wedge\d y=0$. On the other hand, $\d\alpha_2=\kappa\wedge\alpha_1$ gives 
	$(\gamma_1+\gamma_3)\wedge\beta=0$. Using again that $\beta,\d y$ are pointwise  linearly independent, we conclude $\gamma_1+\gamma_3=0$, hence $\alpha_1+\kappa=O(y^1)$. 
\end{proof}

Locally, one can achieve an even better normal form for the coframe. 

\begin{lemma}\label{lem:2}
	Let $\gz$ be a hyperbolic $0$-metric on $\Sigma$. For every $m\in \Sigma$ there exists an adapted oriented orthonormal coframe $\alpha_1,\alpha_2$  on some open neighborhood of $m$ 
	such that the associated spin connection is 
	\[ \kappa=-\alpha_1.\]
	Equivalently, the connection 1-form \eqref{eq:A}  is upper triangular. 
\end{lemma}
\begin{proof}
	Begin by choosing  any adapted oriented orthonormal coframe $\alpha_1,\alpha_2$ for $\gz$, and let $A$ be the associated  connection 1-form. 
	
	Consider the case that $m$ is an interior point. Let $U\subset \Sigma$ be a simply connected open neighborhood of $m$. 
	Since $A$ is flat, it determines a parallel transport $g\in C^\infty(U,\SL(2,\R))$, i.e. $A=g\bullet 0=-(\d g)g^{-1}$ with initial condition $g(m)=e$. The Iwasawa decomposition for $\SL(2,\R)$ gives a map $R(\psi)\colon U\to \SO(2)$ such that 
	$R(\psi) g$ is upper triangular. It hence follows that $ R(\psi)\bullet A=(R(\psi)g)\bullet 0$ is upper triangular. 
	
	The  case that $m$ is a boundary point requires more care, since we need a coframe rotation that extends all the way to the boundary. 
	 Pick adapted local coordinates $x,y$ on a simply connected open neighborhood $U\subset \Sigma$ of $m$, and write $\alpha_1=\f{1}{y} \beta+O(y^0)$ as in the proof of the previous lemma. Then
	\[ A=  \left(\begin{array}{cc}   \f{1}{2y}\ \d y+O(y^0) &    \f{1}{y}\beta+O(y^0)\\  O(y^1) & -\f{1}{2y}\ \d y+O(y^0)\end{array}\right)
	=\left(\begin{array}{cc}  y^{-1/2}& 0\\  0 & y^{1/2}\end{array}\right)\bullet A'
	,\]
	where $A'\in \Omega^1(U,\g)$ is a  \emph{regular} flat connection on $U$ -- all matrix entries extend smoothly to the boundary. 
	This determines a parallel transport $g'\colon U\to \SL(2,\R)$  i.e.  
	$A'=g'\bullet 0=-(\d g')g'^{-1}$ with initial condition $g'|_m=e$.  Away from the boundary, we obtain $A=g\bullet 0$ with 
	\begin{equation}\label{eq:paralleltransport}
	g=	\left(\begin{array}{cc}  y^{-1/2}& 0\\  0 & y^{1/2}\end{array}\right) g'\colon \ 
	U-\partial\Sigma\to \SL(2,\R).
	\end{equation}
	By the Iwasawa decomposition for $\SL(2,\R)$
	there is a unique map $R(\psi)\colon U-\partial\Sigma\to \SO(2)$ with the property that $R(\psi) g$ is upper triangular with positive diagonal. With this choice, $R(\psi)\bullet A$ is upper triangular on $U-\partial\Sigma$. To show that $\psi$ extends smoothly to the boundary, write 
	\begin{equation}\label{eq:gequation}
	g=\left(\begin{array}{cc}  a&  b\\  c & d\end{array}\right),\ \ g'=
	\left(\begin{array}{cc}  a'&  b'\\  c' & d'\end{array}\right). 
	\end{equation}
Then
	\[ \psi=-\arctan(c/a)=-\arctan(y c'/a').\]
	Since $a'|_m=1$, this extends smoothly to all of $U$. 
\end{proof}

\begin{remark}
	The parallel transport \eqref{eq:paralleltransport}, as a map into $G=\PSL(2,\R)$, is well-defined only away from the boundary. 
	It becomes well-defined up to the boundary if it is regarded as a map to the `wonderful' compactification $\ol{G}$. 
\end{remark}

\subsection{Local normal form} 
Lemma \ref{lem:2} allows us to give a quick proof of the local normal form for hyperbolic 0-metrics. Earlier proofs proceed 
through uniformization and the classification of ends \cite[Theorem 2.3]{bor:spec}, or through estimates for sectional curvatures
\cite{gui:pol}. (We thank D. Borthwick for these references.) 

\begin{theorem}\label{th:normalform}
	Let $\mathsf{g}$ be a hyperbolic $0$-metric on $\Sigma$. Every $m\in \Sigma$ admits an open neighborhood  $U$ and a $0$-isometry $U\to \ol{\HH}$. 
\end{theorem}	

\begin{proof}
	Consider the case that $m$ is a boundary point. (For interior points the argument is similar.) 	By Lemma \ref{lem:2}, we may choose 
	an adapted 
	oriented orthonormal coframe $\alpha_1,\alpha_2$ with $\kappa=-\alpha_1$. The second structure equation \eqref{eq:structureequations} gives
	$\d\alpha_2=-\kappa\wedge\alpha_1=0$. Since the coframe is adapted, the difference 
	$\alpha_2-\f{\d \varrho}{\varrho}$ extends smoothly to the boundary, and hence may be written as 
	$\d f$ near $m$. Hence, taking $y=e^f \varrho$ as a coordinate near $m$, we obtain 
	$\alpha_2=y^{-1}\d y$. 
	The first structure equation shows that $y\alpha_1$ is closed: $\d(y\alpha_1)=\d y\wedge\alpha_1-y\kappa\wedge\alpha_2=0$. We may therefore choose the coordinate $x$ near $m$ so that 
	$y\alpha_1=\d x$, which gives $\alpha_1=y^{-1}\ \d x$. 
	This proves the existence of an isometric diffeomorphism from an open neighborhood $U$ of $m\in \Sigma$ onto an open subset of $\bar{\HH}$.
\end{proof}


\section{Hyperbolic structures from flat connections}
The symplectic structure on the infinite-dimensional Teichm\"uller space $\on{Teich}(\Sigma)$ will be obtained 
by a reduction procedure, starting from the usual Atiyah-Bott symplectic structure on a space of connections.  We encountered 
flat connection 1-forms in the coframe formalism, see \eqref{eq:A}. 
Note however that these 1-forms become singular at the boundary. While it is possible to work with these singular connections, we will pursue a different approach where \eqref{eq:A} represents an \emph{ordinary} connection 1-form $\theta$ on a principal bundle. This is motivated by the theory of geometric structures (cf.~\cite{gol:geo0,gol:geo,thu:3dim}).

Throughout, we take $G=\PSL(2,\R)$, with the action on $\ol{\DD}$ regarded as 
$\ol{\DD}=\ol{\HH}\cup\{\infty\}$, or equivalently via 
$\PSL(2,\R)\cong \on{PSU}(1,1)$ (cf. \eqref{eq:psu11}).  

\subsection{Flat bundles from hyperbolic structures}
In Definition \ref{def:hyperbolic}, hyperbolic structures on surfaces $\Sigma$ with boundary were described in terms of charts  $\phi_\alpha \colon U_\alpha\to \ol{\DD}$ with transition functions $h_{\alpha\beta}\in G$, 
i.e. $\phi_\alpha(x)=h_{\alpha\beta}. \phi_\beta(x)$ on $U_\alpha\cap U_\beta$.  The transition functions define a principal $G$-bundle 
\[ \pi\colon P\to \Sigma,\] 
obtained from $\bigsqcup_\alpha (U_\alpha\times G)$ by identifying $(x,g)\in U_\beta\times G$ with 
$(x,h_{\alpha\beta}\,g)\in U_\alpha\times G$. 
The charts themselves determine a $G$-equivariant morphism of manifolds with boundary\footnote{A \emph{morphism of manifolds with boundary} $F\colon M_1\to M_2$ is a smooth map with the property that the pullback of a boundary defining function on $M_2$ is a boundary defining function on $M_1$. Note that such a map determines a morphism of the 0-tangent bundles.}
\[ \sigma\colon P\to \ol{\DD}\] 
given in the local trivializations by $U_\alpha\times G\to \ol{\DD},\ (x,g)\mapsto g^{-1}. \phi_\alpha(x)$. We refer to $\sigma$ as a
\emph{developing section}, since it may be regarded as a section of the associated bundle with fiber $\ol{\DD}$.  

The principal bundle $P$ comes equipped with a flat connection $\theta\in \A_{\on{flat}}(P)$, given in the defining local trivializations by $A_\alpha=0$.  It has the following special property: Let $\on{At}(P)=TP/G$ be the Atiyah algebroid (see Appendix \ref{subsec:atiyah}), and 
denote by $j^\theta\colon T\Sigma\to \At(P)$ the horizontal lift defined by $\theta$. 
Then the composition 
\begin{equation}\label{eq:positivitycondition} T\Sigma\stackrel{j^\theta}{\lra} \At(P)\stackrel{T\sigma}{\lra} V_\sigma=(\sigma^*T\ol{\DD})/G\end{equation}
is an \emph{orientation preserving bundle isomorphism}. In terms of the local trivialization $P|_{U_\alpha}=U_\alpha\times G$, we have  $V_\sigma|_{U_\alpha}=\phi_\alpha^*T\ol{\DD}$, and 
\eqref{eq:positivitycondition} 
 is just the tangent map $T\phi_\alpha\colon TU_\alpha\to V_\sigma|_{U_\alpha}=\phi_\alpha^*T\ol{\DD}$.

\begin{remark}
If the surface is the closed Poincar\'{e} disk itself, these constructions become tautological: The principal bundle is the trivial bundle 
$\ol{\DD}\times G$ with the trivial connection $\pr_2^*\theta^L$ (where $\theta^L$ is the left-invariant Maurer-Cartan form), 	
and $\sigma(z,g)=g^{-1}. z$. All of these data are equivariant for the $G$-action on $\ol{\DD}$. 
More generally, if $\Sigma=(\ol{\DD}-\mathfrak{L})/\Gamma$ as in Remark \ref{rem:schottky},  the triple $(P,\sigma,\theta)$ for $\Sigma$ is obtained from the corresponding triple for $\ol{\DD}-\mathfrak{L}$, by taking the quotient under $\Gamma$. 
\end{remark}

Similarly, any projective structure on an oriented 1-manifold $\CC$ (Definition \ref{def:projective}) determines a principal $G$-bundle $Q\to \CC$ with a developing section $\tau\colon Q\to \p\DD$ and a connection $\vartheta\in \Omega^1(Q,\g)$ such that the composition of maps 
\begin{equation}\label{eq:positivitycondition2} T\CC \stackrel{j^\vartheta}{\lra} \At(Q)\stackrel{T\tau}{\lra} V_\tau=(\tau^*T\p \DD)/G\end{equation}
is an orientation preserving isomorphism. If $\CC=\p\Sigma$, and the projective structure on $\CC$ 
is induced by a hyperbolic structure on $\Sigma$, then $Q=\p P$ is the restriction of $P$, with $\tau=\p\sigma$ the restriction of $\sigma$,  and with  
connection 1-form $\vartheta=\p\theta$ the pullback of $\theta$.

\subsection{Hyperbolic structures from flat bundles}
We shall now reverse the procedure and take as our starting point the data 
\begin{equation}\label{eq:psigma}
\xymatrix{
	P  \ar[r]^{\sigma}\ar[d] &\ol{\DD}\\
	\Sigma & 
}
\end{equation}

of a principal $G$-bundle and a \emph{developing section} $\sigma$ (a  $G$-equivariant  morphism of manifolds with boundary).
Over the interior $\on{int}(\Sigma)$, the map $\sigma$ takes values in $\DD=G/K$ (where $K\cong \SO(2)$ is the stabilizer of 
$i\in \HH\cong \DD$), and so defines a reduction of structure group 
\begin{equation}\label{eq:pk} P_K\subset P|_{\on{int}(\Sigma)}.\end{equation}
On the other hand, the boundary restriction 
$\partial\sigma\colon \partial P=P|_{\partial\Sigma}\to \p\DD$
takes values in $\partial\DD=G/B^-$, where $B^-$ is the stabilizer of 
$0\in \partial\HH\cup\{\infty\} \cong\partial \DD$, and so defines a reduction of structure group 
\begin{equation}\label{eq:pb}
	 (\partial P)_{B^-}\subset \partial P.\end{equation}
Note that $B^-\subset G=\PSL(2,\R)$ is the image of the group of lower triangular matrices with positive diagonal entries. 
It is isomorphic to  $\R\rtimes \R_{>0}$; in particular it is \emph{contractible}. 
The role of $\sigma$ is to combine these two reductions of structure group: to $K$ over the  interior, and to $B^-$ over the boundary.


%
\begin{definition}
	A connection $\theta\in \A(P)$ is  called $\sigma$-\emph{positive} (or simply \emph{positive}, if $\sigma$ is understood)  if the map $T\Sigma\to V_\sigma$ given in 
	\eqref{eq:positivitycondition} is an orientation preserving isomorphism. 
\end{definition}
Denote by $\A^{\on{pos}}(P)$ the space of positive connections, and by $\A^{\on{pos}}_{\on{flat}}(P)$ those which are furthermore flat. 

 
There is a natural map from $\A^{\on{pos}}(P)$ to the space of 0-metrics: The standard 0-metric on $T\ol{\DD}$ gives a 
0-metric on $V_\sigma=\sigma^* T\ol{\DD}/G$; a 
positive connection gives an isomorphism 
$T\Sigma\cong V_\sigma$. 

\begin{proposition}\label{prop:flat=hyperbolic}
If $\theta\in \A^{\on{pos}}(P)$ is flat, then the 0-metric $\gz$ defined by $\theta$ is hyperbolic. 
\end{proposition}
\begin{proof}
Choose  local trivializations $P|_{U_\alpha}\cong U_\alpha\times G$ taking $\theta$ to the trivial connection, and write 
$\sigma(x,g)=g^{-1}. \phi_\alpha(x)$ in terms of the trivialization. The positivity condition ensures that the maps $\phi_\alpha\colon U_\alpha\to \ol{\DD}$ are  orientation preserving diffeomorphisms onto their image, and so define a hyperbolic structure. Clearly, $\gz$ is the 0-metric associated to 
this hyperbolic structure. 
\end{proof}

\begin{remark}\label{rem:leaves}
The construction may also be understood as follows: A flat connection $\theta$ determines a horizontal foliation of $P$. 
Positivity means exactly that $\sigma$ restricts to orientation-preserving local diffeomorphisms from the horizontal leaves 
to $\ol{\DD}$. Hence, the hyperbolic structure on $\ol{\DD}$ pulls back to a $G$-invariant 
hyperbolic structure on the horizontal foliation, which then descends to $\Sigma$.  
\end{remark}

Given an oriented 1-manifold $\CC$, we may similarly consider the data 
\begin{equation}\label{eq:qtau}
\xymatrix{
	Q  \ar[r]^{\tau}\ar[d] &\p{\DD}\\
	\CC & 
}
\end{equation}
of a principal $G$-bundle over $\CC$ with a developing section (a $G$-equivariant map to $\p\DD\cong \RP(1)$). A connection $\vartheta$ on $Q$ is called positive if the 
map \eqref{eq:positivitycondition2} is an orientation preserving isomorphism. Such a connection defines a projective structure on $\CC$; conversely, every projective structure on $\CC$ arises in this way. 

Returning to the pair $(P,\sigma)$ for surfaces with boundary, we have:

\begin{proposition}	
	A connection $\theta\in \A(P)$ satisfies the $\sigma$-positivity condition 
	along the boundary $\pS$  if and only if the pullback 	
	connection $\partial\theta\in \A(\partial P)$ is $\partial\sigma$-positive. 
\end{proposition}
\begin{proof}
	The map $T\Sigma\to V_\sigma$	given by $\theta$ restricts to the map  $T\pS\to V_{\p\sigma}$ given by $\partial\theta$. The resulting map 
	on quotients, \[ \nu(\Sigma,\pS)\to V_\sigma|_{\p\Sigma}/V_{\p\sigma},\] does not depend on the choice of $\theta$. In fact,  it is simply 
	the map obtained by applying the normal bundle functor to the map of pairs $\sigma\colon (P,\p P)\to (\DD,\p\DD)$, using that 
	\[ \nu(\Sigma,\partial\Sigma)=\nu(P,\partial P)/G,\ \ \ \ V_\sigma|_{\p\Sigma}/V_{\p\sigma}=	(\p\sigma)^*\nu(\DD,\partial\DD)/G.\] 
	In particular, the map on quotients is always an orientation preserving isomorphism. We conclude that $T\Sigma|_{\pS}\to V_\sigma|_{\pS}$	is an orientation preserving isomorphism if and only if $T\pS\to V_{\p\sigma}$ is an  orientation preserving isomorphism. 
\end{proof}

\subsection{Relationship with coframe formalism}\label{subsec:relation}
Given $(P,\sigma)$, consider the reduction of structure group \eqref{eq:pk} to $K\subset G$. A trivialization of $P_K$  
over $U\subset \on{int}(\Sigma)$ determines a trivialization $P|_U=U\times G$ such that $\sigma(m,g)=g^{-1}. i$. 
Let $A\in \Omega^1(U,\g)$ be the connection 1-form of $\theta$ in this trivialization. Define 1-forms $\alpha_1,\alpha_2$, and $\kappa$ by writing 
	\begin{equation} \label{eq:A2} A=\hh\left(\begin{array}{cc}\alpha_2&\alpha_1-\kappa\\ \alpha_1+\kappa&-\alpha_2
\end{array}
\right).\end{equation}

\begin{proposition}\label{lem:posA}
The connection $\theta\in \A(P)$ is positive over $U$ if and only if  $\alpha_1,\alpha_2$ are an oriented coframe. 
 In this case, $\alpha_1,\alpha_2$ is an orthonormal coframe for the metric $\gz$ defined by $\theta$; if the connection is flat then $\kappa$ is the  spin connection for this coframe. 
\end{proposition}	
\begin{proof}
Since $T\DD\cong T\HH=T(G/K)=G\times_K \k^\perp$, we have 
\[ V_\sigma|_{\on{int}(\Sigma)}=(P_K\times \k^\perp)/K=\k^\perp(P_K).\] 
Hence, $V_\sigma|_U=U\times \k^\perp$, and \eqref{eq:positivitycondition} 
becomes a map 
\begin{equation}\label{eq:positivity3} TU\to U\times \k^\perp.\end{equation}
Viewed as an element of $\Omega^1(U,\k^\perp)$, this map is the symmetric part of the connection 1-form $A$. The condition that  \eqref{eq:positivity3} is an orientation preserving isomorphism means exactly that $\alpha_1,\alpha_2$ is an oriented orthonormal coframe. Finally, the metric on $V_\sigma|_U$ corresponds to the standard metric on $\k^\perp=T_i\HH$, and the 
metric on $TU$ induced by \eqref{eq:positivity3} is exactly the one defined by the coframe $\alpha_1,\alpha_2$. 
\end{proof}
%

The reduction of structure group to $K$ does not extend to the boundary. To describe the limiting behaviour, we shall work with the following lemma. 
Let $\varrho\in C^\infty(\Sigma)$ be a boundary defining function. 

\begin{lemma}[Normal form at boundary] \label{lem:normalform}
	Given $(P,\sigma)$, there exists $\epsilon>0$ and a trivialization $P|_U=U\times G$ over $U=\varrho^{-1}\big([0,\epsilon)\big)$ such that, in terms of the trivialization,  
	\[ \sigma(m,g)=g^{-1}. (i\varrho(m)).\] 	
\end{lemma}
As usual, we identify $\ol{\DD}\cong \ol{\HH}\cup\{\infty\}$; thus $i\varrho(m)$ is regarded as an element of the closed upper half plane.

\begin{proof} 
	The choice of a trivialization of $\sigma^{-1}(0)=(\p P)_{B^-}$ gives a trivialization $\p P=\pS\times G$ such that 
	$(\p\sigma)(m,g)=g^{-1}. 0$ for $m\in\pS$. Extend it to a trivialization of $P$ over 
	$U=\varrho^{-1}\big([0,\epsilon)\big)$ for some $\epsilon>0$.
	In terms of this trivialization, $\sigma$ is of the form 
	\[ \sigma(m,g)=g^{-1}. f(m),\] 
	where $f\colon U\to \ol{\DD}\cong \ol{\HH}\cup\{\infty\}$ 
	is a morphism of manifolds with boundary, with $f|_{\pS}=0$. Taking $\epsilon$ smaller if needed, we may assume $f$ takes values in $\ol{\HH}\subset \ol{\DD}$.  In particular,  the imaginary part $\on{Im}(f)$ is a boundary defining function, and so 
	$\varrho=u^2\, \on{Im}(f)$ for some function $u\in C^\infty(U,\R_{>0})$. We have 
			\[ \left[\begin{array}{cc}
		u& 0\\ 0 & u^{-1}
		\end{array}\right] \left[\begin{array}{cc}
		1& -\on{Re}(f)\\ 0 & 1
		\end{array}\right]. f= \left[\begin{array}{cc}
		u& 0\\ 0 & u^{-1}
		\end{array}\right]
		. i\on{Im}(f)
		=	i\varrho
		\]
	Hence, a gauge transformation by $\left[\begin{array}{cc}
	u& 0\\ 0 & u^{-1}
	\end{array}\right] \left[\begin{array}{cc}
	1& -\on{Re}(f)\\ 0 & 1
	\end{array}\right]$ replaces $\sigma(m,g)=g^{-1}. f(m)$ with $g^{-1}. i\varrho(m)$. 
\end{proof} 

Given a connection $\theta\in \A(P)$, let $A\in \Omega^1(U,\g)$ be its connection 1-form in terms of the 
 trivialization from this lemma. Over $U-\partial\Sigma$, the gauge transformation by 
\begin{equation}\label{eq:singgauge}
h=\left[\begin{array}{cc}\rho^{1/2}& 0 \\ 0 & \rho^{-1/2}\end{array}\right] \colon U-\pS\to G\end{equation}
is defined, and satisfies $h^{-1}. i\varrho(m)=i$. Hence, $(h^{-1}. \sigma)(m,g)=g^{-1}. i$, and 
$h^{-1}\bullet A$ is of  the form \eqref{eq:A2}, defining $\alpha_1,\alpha_2,\kappa$.  
As we saw above, the positivity condition 
on $\theta$ means that over the interior, $\alpha_1,\alpha_2$ are an oriented orthonormal coframe. The original connection 1-form is expressed in terms of these data as
\[ A=\hh\left(\begin{array}{cc}\alpha_2-\varrho^{-1}\d\varrho&\varrho(\alpha_1-\kappa)\\ \varrho^{-1}(\alpha_1+\kappa)&-\alpha_2
+\varrho^{-1}\d\varrho
\end{array}
\right).\]
Since $A$ is a \emph{regular} connection 1-form, this shows that $\alpha_1,\alpha_2$ extend to elements  of $\zz\Omega^1(U)$,
and define an adapted orthonormal coframe.

\subsection{Existence of positive connections}

Recall the classification of principal bundles over compact, connected, oriented surfaces
$\Sigma$ with boundary. Let $H$ be a connected Lie group, and $R\to \Sigma$ a principal $H$-bundle  with a given homotopy class of trivializations (framings) of $R|_{\pS}$. Pick $x_0\in \on{int}(\Sigma)$, and choose a trivialization of $R$ over  the punctured surface $\Sigma-\{x_0\}$ such that the trivialization along the boundary is in the prescribed class. 
Also choose a trivialization of $R$ over an  embedded disk $D\subset \on{int}(\Sigma)$ around $x_0$. The 
homotopy class of the transition map $D-\{x_0\}\to H$ defines an element 
\[ \mathsf{e}(R)\in \pi_1(H).\]
If $R_1,R_2$ are two principal $H$-bundles with homotopy classes of trivializations along the boundary, we have $\ez(R_1)=\ez(R_2)$  if and only if there exists a bundle isomorphism $R_1\to R_2$ which intertwines the homotopy classes of 
trivializations over the boundary. In particular, taking $R=\on{Fr}_{\SO(2)}(\Sigma)$ to be the 
oriented orthonormal frame bundle for a Riemannian metric, with its standard trivialization along the boundary (where the first element of a frame is tangent to the boundary, pointing in the positive direction), the element $\mathsf{e}(R)\in \pi_1(\SO(2))=\Z$ is the Euler characteristic $\chi(\Sigma)$ of the surface. Given a principal $G$-bundle $P\to \Sigma$ with developing section $\sigma\colon P\to \ol{\DD}=\ol{\HH}\cup\{\infty\}$ as above, there is a distinguished homotopy class of trivializations along the boundary -- those trivializations for which $\sigma(m,g)=g^{-1}. 0$ for $m\in \pS$. 
Let 
\[ \ez(P,\sigma)\in \pi_1(G)=\Z\] 
be the resulting invariant. 

\begin{proposition}\label{prop:collar}\label{prop:quantumnumber}
	Given two pairs $(P,\sigma)$ and $(P',\sigma')$, we have	$\ez(P,\sigma)=\ez(P',\sigma')$ if and only if there exists an isomorphism 
	$P\to P'$ taking $\sigma$ to $\sigma'$ and inducing the identity on the base. 	
\end{proposition}
\begin{proof}
	The necessity of the condition is obvious. To show that it is also sufficient, suppose $\ez(P,\sigma)=\ez(P',\sigma')$. 
	We may assume $P=P'$, and that $\sigma,\sigma'$ define the same homotopy class of trivializations of $\p P=P|_{\pS}$. 
	 Using the normal form near the boundary (Lemma \ref{lem:normalform}), we may assume that $\sigma,\sigma'$ coincide over an open neighborhood of the boundary. 	Over the interior, $\sigma,\sigma'$ may be regarded as sections of the associated bundle with fiber $\DD$, which agree near the  boundary.  As is well-known, given any two distinct points $z,z'\in \DD$ there is a unique element $\phi(z,z')\in G$ taking $z$ to $z'$ and preserving the geodesic through $z,z'$. It extents smoothly to a map $\phi\colon \DD\times \DD\to G$ with $\phi(g.z,g.z')=g\phi(z,z') g^{-1}$.  
	 Consequently,we obtain a gauge transformation $h\in \Gau(P)$ taking $\sigma$ to $\sigma'$.  
	 This gauge 
	 transformation is trivial near the boundary since $\sigma,\sigma'$ agree there.
\end{proof}


The existence of $\sigma$-positive connections places a topological condition on $(P,\sigma)$.

\begin{proposition}\label{prop:eulerorelse}
	Let $\Sigma$ be a compact, connected, oriented surface with boundary, and $P\to \Sigma$ a principal $G$-bundle with a developing section $\sigma\colon P\to \ol{\DD}$.  Then the space of positive connections is empty unless $\mathsf{e}(P,\sigma)=\chi(\Sigma)$. 
\end{proposition}
\begin{proof}
	The oriented rank 2 bundle $V_\sigma|_{\p\Sigma}$  has a distinguished rank 1 subbundle  $V_{\p\sigma}$; hence there is 
	a unique homotopy class of trivializations  of $V_\sigma$ along the boundary, taking the subbundle to  
	$\R\oplus 0\subset \R^2$. 
	Letting $\on{Fr}_{\SO(2)}(V_\sigma)$ be the frame bundle for (any) fiber metric, the invariant $\ez(\on{Fr}_{\SO(2)}(V_\sigma))\in \pi_1(\SO(2))=\Z$ is defined. We claim that 
\begin{equation}\label{eq:obvious}	\ez(\on{Fr}_{\SO(2)}(V_\sigma))=\ez(P,\sigma).\end{equation}
To see this, choose a covering of $\Sigma$, consisting of an open subset $U_1=\Sigma-\{x_0\}$ where $x_0\in \on{int}(\Sigma)$, and 
an open neighborhood $U_2$ of $x_0$, contained in the interior of $\Sigma$ and diffeomorphic to an open disk. Choose a trivialization 
of $P$ over $U_1$, inducing the given class of trivializations along the boundary, and choose also a trivialization over $U_2$. 
Let $f_i\colon U_i\to \ol{\DD}$ be the maps describing $\sigma$ in these trivializations. We may arrange that $f_1,f_2$ are both constant 
(equal to $i$) over $U_1\cap U_2$. Then the transition map $U_2-\{x_0\}=U_1\cap U_2\to G$ takes values in $K$. The trivializations of 
$P|_{U_i}$ also trivialize $V_\sigma|_{U_i}$, and the transition map for its frame bundle agrees with that for $P$ under the isomorphism $K\cong \SO(2)$. This proves \eqref{eq:obvious}. 

A positive connection determines an oriented vector bundle isomorphism $T\Sigma\to V_\sigma$, and hence 
\[ \chi(\Sigma)=\ez(\on{Fr}_{\SO(2)}(T\Sigma))=\ez(\on{Fr}_{\SO(2)}(V_\sigma))=\ez(P,\sigma). \qedhere \]
\end{proof}

\section{Automorphisms}In this section, we discuss the structure of the groups of  automorphisms preserving a given developing section, for $G$-bundles over surfaces and over curves. 
Recall that $B^-\subset G=\PSL(2,\R)$ is the image of lower triangular matrices with positive diagonal entries. Thus $B^-=AN^-$ 
where $A$ is the image of positive diagonal matrices, and $N^-=[B^-,B^-]$ is the image of lower triangular matrices with $1$'s on the diagonal.

\subsection{The group $\Aut(Q,\tau)$}
Let $Q\to \CC$ be a principal $G$-bundle over a compact oriented 1-manifold, and $\tau\colon Q\to \p\DD
=\p\HH\cup\{\infty\}
$ a developing section. 

\begin{proposition}
The group $\Gau(Q,\tau)$  of gauge transformations preserving $\tau$
is the group of sections of a group bundle $G(Q,\tau)\to \CC$, with typical fiber $B^-=AN^-$.  
It fits into an exact sequence with the  group  $\Aut(Q,\tau)$ of 
automorphisms of preserving $\tau$,
\[ 1\to \Gau(Q,\tau)\to \Aut(Q,\tau)\to \Diff(\CC)\to 1.\] 
Infinitesimally, $\gau(Q,\tau)\subset \aut(Q,\tau)$ are the sections of a Lie algebroids $\g(Q,\tau)\subset \At(Q,\tau)$, 
described as the kernel of the bundle maps $\g(Q)\subset \At(Q)\to V_\tau$. 
We have  an exact sequence of  Lie algebroids 
\[ 0\to \g(Q,\tau)\to \At(Q,\tau)\to T\CC\to 0.\]	
\end{proposition}
\begin{proof}
	The developing section defines a reduction of structure group $Q_{B^-}=\tau^{-1}(0)\subset Q$, and the groups 
	$\Gau(Q,\tau)\subset \Aut(Q,\tau)$ are identified with the gauge transformations and automorphisms of 
	$Q_{B^-}$. (Every $B^-$-equivariant diffeomorphism of $Q_{B^-}$ extends uniquely to a $G$-equivariant diffeomorphism of 	$Q$; the latter preserves $\tau$.) In particular, $\At(Q,\tau)$ is just the Atiyah algebroid of $Q_{B^-}$. The sections of 
	$\At(Q,\tau)$ are identified with the $B^-$-invariant vector fields on $Q_{B^-}$, or equivalently with the 
	$G$-invariant vector fields on $Q$ that are $\tau$-related to $0$. Equivalently, this is the kernel of 
	the bundle map to $V_\tau$. 
\end{proof}

\subsection{The group $\Aut(P,\sigma)$}\label{subsec:gauge}
We now give a similar discussion for principal bundles over oriented surfaces $\Sigma$  with boundary. 
Let$(P,\sigma)$ as in \eqref{eq:psigma}. As we saw, $\sigma$ gives a simultaneous description of two reductions of structure group: Over the interior, the structure group of $P$ is reduced to $K$, over $\pS$ it is reduced to $B^-$. 

\begin{proposition}\label{prop:generalstructure}
	The groups of gauge transformations and automorphism of $P$ preserving $\sigma$ fit into an exact sequence 
	\[ 1\to \Gau(P,\sigma)\to \bb\Aut(P,\sigma)\to \bb\Diff(\Sigma)\to 1.\] 
	The group $\Gau(P,\sigma)$ is the group of sections of a family of Lie groups
	 $G(P,\sigma)$, with typical fibers $K$ over interior points 	and $N^-$ at boundary points; restriction to the boundary identifies
		\[ G(P,\sigma)|_{\pS}
	=[G(\p P,\p\sigma),G(\p P,\p\sigma)].\]
\end{proposition}
(A \emph{family of Lie groups} is a Lie groupoid for which source and target map coincide. It need not be locally trivial.)
\begin{proof}
Over the interior of $\Sigma$, 	 the developing section defines a reduction of the structure group to $K\subset G$, 
and the groups 
$\Gau(P|_{\on{int}(\Sigma)},\sigma)\subset \Aut(P|_{\on{int}(\Sigma)},\sigma)$ are identified with the gauge transformations and automorphisms of $P_K$. The main task in proving Proposition \ref{prop:generalstructure} is to understand the behavior near 
the boundary.  We shall use the local normal form, Lemma \ref{lem:normalform}. 
Thus let 
\[U=\partial\Sigma\times [0,\epsilon),\ \ P|_U=U\times G,\]
with $\sigma(m,g)=g^{-1}. (i\varrho(m))$. 
Denote the points of $U$ by $m=(x,y)$, so that $\varrho(m)=y$. 
The automorphisms of $P|_U$ may be written as pairs $(h,\Phi)$, where $\Phi\in \bb\Diff(U)$ and 
$h\in \Gau(P)=C^\infty(U,G)$. Such an automorphism preserves the developing section $\sigma(m,g)=g^{-1}.f(m)$ if and only if 
\[ h(m).f(\Phi^{-1}(m))=f(m).\] 
In our case, $f(x,y)=iy$, we have: 
\begin{lemma}\label{lem:normalform2}
The elements of $\Gau(P|_U,\sigma)$ with compact support in $U$ are of the form 
\[\exp\Big(
\begin{array}{cc} 0 & -\chi \varrho^2\\ \chi & 0
\end{array}\Big)\]
for compactly supported $\chi\in C^\infty(U)$. Every element of 
$\bb\!\Aut(P|_U,\sigma)$ whose base map is compactly supported in  $U$ 
is uniquely the product of such a gauge transformation  and 
an automorphism 
\[ \left(\Big[
\begin{array}{cc} e^{-\lambda/2}& 0\\ 0 & e^{\lambda/2}
\end{array}\Big],\ \Phi\right)\]
where $\Phi\in\bb\Diff(U)$ has compact support in $U$, and $\lambda\in C^\infty(U)$ is the compactly supported function defined by $\Phi_*\varrho=e^\lambda\,\varrho$. 
\end{lemma}

\begin{proof}
A gauge transformation $h\in C^\infty(U,G)$ fixes $\sigma$ if and only if $h(x,y)\in G_{iy}$ for all $y$. 
For $y>0$,  the stabilizer of $iy\in \HH$ in $G$ is 
\begin{equation}\label{eq:stab} 
G_{iy}=\left\{\exp\Big(\begin{array}{cc} 0 & -t\,y^2\\ t & 0
\end{array}\Big)\Big|\ t\in\R\right\}\cong K\end{equation}
%
This fits uniquely into a smooth family of subgroups of $\{(x,y)\}\times G$ for all $y\ge 0$, by taking the fiber for $y=0$ to be 
$N^-$. A smooth gauge transformation fixing $\sigma$ must take values in this family of Lie groups. 
Consider next a compactly supported diffeomorphism $\Phi\in \bb\Diff_+(U)$. The push-forward $\Phi_*\varrho=\varrho\circ \Phi^{-1}$ is again a boundary defining function, and so is of the form $e^\lambda\varrho$. The hyperbolic transformation 
given by the diagonal matrix with entries $e^{-\lambda(m)/2},e^{\lambda(m)/2}$ down the diagonal takes 
$f(\Phi^{-1}(m))=i\,e^{\lambda(m)}  y$ back to $iy$. 
\end{proof}
The Lemma (and its proof) verify that $\Gau(P,\sigma)$ are the sections of a family of Lie groups $G(P,\sigma)$. It also shows that
every diffeomorphism in $\bb\Diff(\Sigma)$ lifts to an automorphism in 
$\bb\Aut(P,\sigma)\to \bb\Diff(\Sigma)$: For diffeomorphisms  
supported in the interior of $\Sigma$ this is done by lifting to an automorphisms of $P_K\subset P_{\on{int}(\Sigma)}$; for 
diffeomorphisms supported in a collar neighborhood of the boundary the Lemma gives an explicit lift. 
\end{proof}

We see in particular that the restriction map $\Gau(P,\sigma)\to \Gau(\p P,\p\sigma)$ is \emph{not} surjective. On the other hand, we have: 
\begin{proposition}\label{prop:restrictionsurjective}
	The restriction map 
	\[ \bb\!\!\Aut_\oz(P,\sigma)\to \Aut_\oz(\p P,\p\sigma)\] 
	is surjective. In fact, every 
	element of $ \Aut_\oz(\p P,\p\sigma)$ admits an extension to an element of $\bb\!\Aut_\oz(P,\sigma)$ which is supported on a collar neighborhood of the boundary. 
\end{proposition}
\begin{proof}
	We work with the normal form (Lemma \ref{lem:normalform}) over a collar neighborhood $U$ of the boundary.  	
	Given an element of $\Aut_\oz(\p P,\p\sigma)$, with base map $\p\Phi\in \Diff_\oz(\p\Sigma)$, choose an extension to a diffeomorphism $\Phi$ with compact support on $U$. Lemma \ref{lem:normalform2} shows how to lift 
	$\Phi$ to an element of $\bb\!\Aut_\oz(P,\sigma)$, extending the given automorphism along the boundary. 
\end{proof}

For the sake of completeness, we also give the infinitesimal descriptions. Recall that the $b$-tangent bundle $\bb TM$ 
of a manifold $M$ with boundary is the vector bundle whose sections are vector fields tangent to the boundary $\p M$. 

\begin{proposition}
	The kernel of bundle map $\At(P)\to V_\sigma$ is a Lie subalgebroid $\bb\At(P,\sigma)$ of the Atiyah algebroid, with
	$\bb\aut(P,\sigma)$ as its space of sections. It fits into an exact sequence of Lie algebroids
	\begin{equation}\label{eq:batiyahalgebroid}
	0\lra \g(P,\sigma)\lra \At(P,\sigma)\lra \bb T\Sigma\lra 0.\end{equation}
	Restriction to the boundary is a  Lie algebroid isomorphism 
	\begin{equation}\label{eq:surj2} \At(P,\sigma)|_{\p\Sigma}\cong \At(\p P,\p\sigma)\end{equation}
	inducing the inclusion $\g(P,\sigma)|_{\pS}\hra \g(\p P,\p\sigma)$. 
\end{proposition}
\begin{proof}
Over the the interior, 	the kernel of $\At(P)\to V_\sigma$ is the subalgebroid $\At(P_K)$ given by the reduction of structure group 
$P_K\subset P|_{\on{int}(\Sigma)}$. Hence, it suffices to study the situation near the boundary. Using the normal form from Lemma \ref{lem:normalform}, we have $P|_U=U\times G$, with $\sigma(m,g)=g^{-1}.f(m)$ for $f(x,y)=iy$.  
The bundle map $\At(P|_U)=TU\times \g\to V_\sigma|_U=f^*T\ol{\HH}$ is given by 
\[ (v,\xi)\mapsto (T_mf)(v)-\xi^\sharp|_{f(m)},\] 
where  $\xi^\sharp$ is the vector field on $\ol{\HH}$ generated by $X$. Since $\xi^\sharp$ is tangent to $\p \HH$, we see that 
for elements $(v,\xi)\in \At(P)|_m,\ m\in\pS$ in the kernel of  the map to $V_\sigma$, the vector $v$ must itself be tangent to $\p\Sigma$. The rest of the discussion is 
as in the proof of Lemma \ref{lem:normalform2}. In particular, infinitesimal gauge transformations are given by functions 
\[ \chi\left(\begin{array}{cc} 0 & -\varrho^2\\ 1 & 0
\end{array}\right)\in C^\infty(U,\g)=\gau(P|_U).\]
Every compactly supported vector field $v\in \bb\Vect(U)$ tangent to $\p\Sigma$ defines a function 
$\lambda=\varrho^{-1}\L_v \varrho$, and the element 
	\[ \left(\Big(
	\begin{array}{cc} -\lambda/2& 0\\ 0 & \lambda/2
	\end{array}\Big),\ v\right)\in C^\infty(U,\g)\times \Vect(U)= \aut(P|_U)
	\]
	is a lift to $\aut(P|_U,\sigma)$. 
\end{proof}

\begin{remark}\label{rem:modelcase}	For the model case $\Sigma=\ol{\DD}$, 
	the Lie algebroid 
	$\At(P,\sigma)$ is identified with the action Lie algebroid $\ol{\DD}\times \g$ , embedded in 
	$\At(P)=T\ol{\DD}\times \g$ by the map $(z,\xi)\mapsto (\xi^\sharp(z),\xi)$. 
\end{remark}

\section{Symplectic structure on $\on{Teich}(\Sigma)$}
We shall now construct a symplectic structure $\omega$ on the Teichm\"uller spaces of hyperbolic structures
on surfaces $\Sigma$ with boundary.   
Throughout, $\Sigma$ will be compact, connected, and oriented, with a given pair $(P,\sigma)$ as in \eqref{eq:psigma}, 
satisfying $\ez(P,\sigma)=\chi(\Sigma)$. 
By Proposition \ref{prop:eulerorelse}, 
the space  $\A^{\on{pos}}(P)$ of $\sigma$-positive connections is non-empty.

\subsection{Hyperbolic metrics from positive connections}
Our starting point is the following description of $\on{Hyp}(\Sigma),\ \on{Teich}(\Sigma)$ as quotients of spaces of flat connections.

\begin{theorem}\label{tr:hypasquotient}
	The space of hyperbolic structures on $\Sigma$ is a quotient, 
	\begin{equation}
	\on{Hyp}(\Sigma)=\A_{\on{flat}}^{\on{pos}}(P)/\Gau(P,\sigma).\end{equation}
	Taking a further quotient by the action of $\zz\Diff_\oz(\Sigma)$, we obtain
	\begin{eqnarray}
	\on{Teich}(\Sigma)&=\A_{\on{flat}}^{\on{pos}}(P)/\zz\!\!\Aut_\oz(P,\sigma).\label{eq:teichred}
	\end{eqnarray}
\end{theorem}

\begin{proof}
	By Proposition \ref{prop:flat=hyperbolic}, every $\theta\in \A^{\on{pos}}_{\on{flat}}(P)$ determines a hyperbolic 0-metric $\gz$ on $\Sigma$; changing $\theta$ by a gauge transformation in $\Gau(P,\sigma)$ does not change $\gz$.  
	Conversely, every hyperbolic 0-metric on $\Sigma$ arises in this way, where $\theta$ is unique up to the action of $\Gau(P,\sigma)$. 
	Indeed, $\gz$ determines a triple $(P',\sigma',\theta)$, with $\ez(P',\sigma')=\chi(\Sigma)$; 	hence $(P',\sigma')$ is related to $(P,\sigma)$ by a bundle isomorphism.  It follows that $\gz$ is defined by a flat connection  $\theta\in \A^{\on{pos}}_{\on{flat}}(P)$ (the image of $\theta'$ under this isomorphism). This proves the description of $\on{Hyp}(\Sigma)$; the remaining assertions are clear. 
\end{proof}

Similarly, given $(Q,\tau)$ as in \eqref{eq:qtau}, we saw that every projective structure on an oriented 1-manifold $\CC$ is obtained from a $\tau$-positive connection on $Q$. 
The latter is unique up to the action of $\Gau(Q,\tau)$.  Hence 
\[ \on{Proj}(\CC)=\A^{\on{pos}}(Q)/\Gau(Q,\tau).\]
The quotient map intertwines the action of 
$\Aut_+(Q,\tau)$ on connections with the action of $\on{Diff}_+(\CC)=\Aut_+(Q,\tau)/\Gau(Q,\tau)$ on projective structures. 
Letting $Q=\p P$ and $\tau=\p\sigma$, we see that the pullback map $\A^{\on{pos}}_{\on{flat}}(P)\to 
\A^{\on{pos}}(\p P)$ descends to the natural map $	\on{Teich}(\Sigma)\to \on{Proj}(\pS)$.

\subsection{Atiyah-Bott symplectic structure}
Let $\cdot$ denote the nondegenerate invariant symmetric bilinear form (`metric')  on $\g=\mf{sl}(2,\R)$ given by 
\begin{equation} \xi\cdot \eta=\on{tr}(\xi\eta).\end{equation} 
It determines a bundle metric on $\g(P)$, and an Atiyah-Bott symplectic structure on the
space $\A(P)$ of connections by 
\[ \omega_{AB}(a,b)=\int_\Sigma a\stackrel{.}{\wedge} b,\] for 
$a,b\in T_\theta\A(P)=\Omega^1(\Sigma,\g(P))$. 
The symplectic structure is invariant under the action of $\bb\!\Aut_+(P)$, and there is a natural affine moment map 
for this action, involving both a bulk term and a boundary term. We refer to Appendix \ref{app:AB} for details. 

For now, we consider the subgroup  $\zz\!\Aut_+(P)$ of automorphisms 
whose base map fixes the boundary. The moment map for this subgroup is given by 
\begin{equation} \A(P)\to \Omega^2(\Sigma,\At(P)^*)\times \A(\p P),\  \ \theta\mapsto 
\big(-F^\theta\cdot s^\theta,\  \p\theta\big).
\end{equation}
Here $F^\theta\in \Omega^2(\Sigma,\g(P))$ is the curvature, and  
\begin{equation}\label{eq:stheta}
 s^\theta\colon \on{At}(P)=TP/G\to \g(P)\end{equation}
is the splitting (`vertical projection') defined by connection $\theta\in \A(P)$.

 One obtains an affine moment map for the action of $\zz\!\Aut_+(P,\sigma)\subset \zz\!\Aut_+(P)$ by projection, replacing $s^\theta$ with $s^\theta|_{\At(P,\sigma)}$. 
  We would like to interpret \eqref{eq:teichred} as a symplectic reduction for this moment map. As we will see, the  boundary term causes some complications. Let us therefore begin with the case $\pS=\emptyset$.

\subsection{Symplectic structure: the case $\pS=\emptyset$}
If the boundary is empty, only the bulk term $-F^\theta\cdot s^\theta$ of the moment map remains. The moment map for $\Aut_+(P,\sigma)$ is  thus given by 
\begin{equation}\label{eq:restrictedmommap}
 \A^{\on{pos}}(P)\to \aut(P,\sigma)^*=\Omega^2(\Sigma,\At(P,\sigma)^*),\ \theta\mapsto -F^\theta\cdot s^\theta|_{\At(P,\sigma)}.\end{equation}

\begin{proposition}\label{prop:casewithoutboundary}
	Let $\Sigma$ be a compact oriented surface without boundary, and $P\to \Sigma$ a principal $G$-bundle with developing section satisfying $\ez(P,\sigma)=\chi(\Sigma)$. Then 
	\[ \on{Teich}(\Sigma)=\A^{\on{pos}}(P)\qu \Aut_\oz(P,\sigma),\]
	a symplectic quotient, with a residual action of $\on{MCG}(\Sigma)$ preserving the symplectic structure.  
\end{proposition}
\begin{proof}
We claim that the zero level set of \eqref{eq:restrictedmommap} is the space $\A^{\on{pos}}_{\on{flat}}(P)$ of flat connections. 
The map ${s^\theta} \colon  \At(P,\sigma)\to  \g(P)$ restricts to the identity on the subbundles $\g(P,\sigma)$, and gives a 
commutative diagram 
\[ \xymatrix{\At(P,\sigma)\ar[d] \ar[r]^{s^\theta} & \g(P)\ar[d]\\
	T\Sigma \ar[r] & V_\sigma
}\]	
where the vertical maps are the quotients maps for the subbundle $\g(P,\sigma)$. For $\theta\in \A^{\on{pos}}(P)$ the lower horizontal map is an isomorphism, hence so is the upper map.  Hence $F^\theta\cdot s^\theta(v)=0$ for all 
$v\in \aut(P,\sigma)=\Gamma(\At(P,\sigma))$ if and only if $F^\theta=0$. Since $\Aut_+(P,\sigma)$ preserves the 
Atiyah-Bott form, the induced action of $\Aut_+(P,\sigma)/\Aut_\oz(P,\sigma)\cong \on{MCG}(\Sigma)$ is again symplectic. 
\end{proof}

	Since the condition $\ez(P)=\chi(\Sigma)$ determines $P$ 
up to isomorphism, the symplectic 2-form $\omega$ on $\on{Teich}(\Sigma)$  does not depend on its choice. 
We will verify in Section \ref{sec:fn} that this symplectic form is the standard Weil-Petersson form, given in Fenchel-Nielsen coordinates by Wolpert's theorem. 	

It is clear from the construction that the residual action of the mapping class group $\on{MCG}(\Sigma)$ preserves the symplectic structure. Hence, we also have 
\[ \M(\Sigma)=\A^{\on{pos}}(P)\qu \Aut_+(P,\sigma).\]

\begin{remark}
	By a classical result, obtained independently by  Goldman 
	\cite{gol:top} and Hitchin \cite{hit:self}, the symplectic structure on Teichm\"uller space $\on{Teich}(\Sigma)$ may be obtained directly as a moduli space of flat connections,  without having to invoke developing sections. That is, 
	\begin{equation} \label{eq:golhit} \on{Teich}(\Sigma)=\A(P)\qu \Gau(P)\end{equation} 
	for any choice of $G$-bundle with $\ez(P)=\chi(\Sigma)$. 
 The proof  of \eqref{eq:golhit} is more involved; we do not know of 
	an independent argument obtaining this result from 
	Proposition \ref{prop:casewithoutboundary}. 
\end{remark}

\subsection{Symplectic structure: the case $\pS\neq \emptyset$}
We now turn to the case of a possibly non-empty boundary. We shall denote $Q=\p P,\ \tau=\p\sigma$.
Equation \eqref{eq:teichred} represents the Teichm\"uller space as a quotient
by the group $\zz\!\Aut_\oz(P,\sigma)$, the identity component of automorphisms whose base map fixes $\pS$. 
However: 

\begin{quote} The pullback of $\omega_{AB}$  to $\A^{\on{pos}}_{\on{flat}}(P)$ does \underline{not} descend to the quotient.  \end{quote}

The problem is that the restriction map $\zz\!\Aut_\oz(P,\sigma)\to \Gau(Q)$ is non-trivial (in fact, it is surjective). Hence,  the boundary terms of the moment map are present. 
Our strategy is to carry out the reduction in stages. Let 
 \[ \Aut_\oz(P,Q,\sigma)\subset \zz\!\Aut_\oz(P,\sigma)
 \]
be the kernel of the restriction map, and put 
%
\begin{equation}\label{eq:hatteich}
 \wh{\on{Teich}}(\Sigma)= \A^{\on{pos}}_{\on{flat}}(P)/\Aut_\oz(P,Q,\sigma)=
\on{Teich}(\Sigma)\times_{\on{Proj}(\pS)} \A^{\on{pos}}(Q). \end{equation}	
Elements of this space are represented by hyperbolic structures on $\Sigma$ together with a lift of the corresponding projective structure on the boundary  to a $\tau$-positive connection on $Q$. 

\begin{lemma}\label{lem:psi}
	$\wh{\on{Teich}}(\Sigma)$ is a symplectic quotient
	$\A^{\on{pos}}(P)\qu \Aut_\oz(P,Q,\sigma)$.  
\end{lemma}
\begin{proof}
The Lie algebra $\aut(P,Q,\sigma)$ is the space of sections of $\At(P,\sigma)$ vanishing along the boundary $\pS$. 
By the same argument as in the proof of Proposition \ref{prop:casewithoutboundary}, if 
$\theta\in  \A^{\on{pos}}(P)$, we have $ F^\theta\cdot s^\theta(v)=0$ for all $v\in \aut(P,Q,\sigma)$ if and only if $F^\theta=0$. 
	Hence, the symplectic quotient by this subgroup is 
	$\wh{\on{Teich}}(\Sigma)$. 
\end{proof}
According to Proposition \ref{prop:restrictionsurjective}, the map $\zz\!\Aut_\oz(P,\sigma)\lra \Gau(Q,\tau)$ is surjective; hence 
$\wh{\on{Teich}}(\Sigma)$ has a residual  action of the group $\Gau(Q,\tau)$, with quotient ${\on{Teich}}(\Sigma)$.
Let 
\[ \A(Q,\tau)=\A(Q)/\on{ann}(\gau(Q,\tau)).\]
This is an affine $\Gau(Q,\tau)$-space, with linear action the coadjoint action on  $\gau(Q,\tau)^*$. 
\begin{lemma}
	The moment map for the action of $\Gau(Q,\tau)$ on $\wh{\on{Teich}}(\Sigma)$ is given by 
	\[ 	\Psi\colon	\wh{\on{Teich}}(\Sigma)\to \A(Q,\tau),\ \ 
	[\theta]\mapsto \ \p\theta\!\!
	\mod \on{ann}(\gau(Q,\tau)).\]
\end{lemma}
\begin{proof}
This follows by reduction, since the boundary term of the moment map for the $\zz\Aut_\oz(P)$-action on $\A(P)$ is $\theta\mapsto \p\theta$. 
\end{proof}
 
Note that $\Psi$ takes values in the subspace $\A^{\on{pos}}(Q,\tau)=\A^{\on{pos}}(Q)/\on{ann}(\gau(Q,\tau))$. 
It turns out that the image of this map is a single coadjoint orbit: 
\begin{lemma}\label{lem:ds}
The action of $\Gau(Q,\tau)$ on  $\A^{\on{pos}}(Q,\tau)$  is free and transitive. 
\end{lemma}
\begin{proof}
This is a well-known fact from Drinfeld-Sokolov theory. We may assume that $Q=S^1\times G$, with $\tau(g,x)=g^{-1}\cdot 0$. The connections on $Q$ are described by their connection 1-forms  $A\in \Omega^1(S^1,\g)$. Write 
\begin{equation}\label{eq:connectiononboundary} A=\left(\begin{array}{cc}
\hh s & a\\ u & -\hh s
\end{array}
\right)\ \d x\end{equation}
with functions $a,s,u\in C^\infty(S^1)$; the connection is $\tau$-positive if and only if $a>0$. 
Taking the quotient by $\on{ann}(\gau(Q,\tau))\cong \Omega^1(S^1,\mf{n^-})$ amounts to omitting the lower left corner; hence $a,s$ serve as parameters on $\A^{\on{pos}}(Q,\tau)$.  There is a unique gauge transformation by an element $h\in C^\infty(S^1,B^-)=\Gau(Q,\tau)$ putting \eqref{eq:connectiononboundary} into Drinfeld-Sokolov normal form, 
that is, having $1$ in the upper right corner and with vanishing diagonal entries. Explicitly,
\[ h=\left(\begin{array}{cc}
1 & 0\\  -\hh s+\hh \f{a'}{a}&1 
\end{array}
\right) \left(\begin{array}{cc}
a^{-\f{1}{2}} & 0 \\ 0  & a^{\f{1}{2}} 
\end{array}
\right).\]
In other words, $h$ is the unique element taking the class of $A$ in $\A^{\on{pos}}(Q,\tau)$ to the base point of $\A^{\on{pos}}(Q,\tau)$ given by $s=0,\ a=1$. 
\end{proof}

Since $\O=\A^{\on{pos}}(Q,\tau)$ is a coadjoint orbit, it has a unique symplectic structure such that the 
$\Gau(Q,\tau)$-action is Hamiltonian, with moment map the inclusion. (See Appendix \ref{app:A1}.) 

By the well-known `shifting trick' from symplectic geometry, the quotient 
$\wh{\on{Teich}}(\Sigma)/\Gau(Q,\tau)$ may be  recast as a symplectic quotient. Let $\O^-$ be the space $\O$ with the opposite symplectic structure. 
\begin{theorem}\label{th:teichhat}
The Teichm\"uller space is a symplectic quotient
\[ \on{Teich}(\Sigma)=(\wh{\on{Teich}}(\Sigma)\times \O^-)\qu \Gau(Q,\tau),\]
where $\wh{\on{Teich}}(\Sigma)=\A^{\on{pos}}(P)\qu \Aut_\oz(P,Q,\sigma)$. 
In particular, $\on{Teich}(\Sigma)$ acquires a symplectic structure. The action of 
$\on{MCG}(\Sigma)$ on  $\on{Teich}(\Sigma)$ preserves the symplectic structure. 
\end{theorem}
\begin{proof}
Only the final claim remains to be proved. The action 
of $\bb\Aut_+(P,\sigma)$ on the space $\A^{\on{pos}}(P)$ preserves the Atiyah-Bott symplectic structure, and restricts to an action on the space of flat connections. We hence obtain a symplectic  action of $ \bb\Aut_+(P,\sigma)/\Aut_\oz(P,Q,\sigma)$ on $\wh{\on{Teich}}(\Sigma)$.  
This contains $\on{MCG}(\Sigma)=\Aut_+(P,Q,\sigma)/\Aut_\oz(P,Q,\sigma)$
and  $\zz\Aut_\oz(P,\sigma)/\Aut_\oz(P,Q,\sigma)=\Gau(Q,\tau)$ as commuting subgroups. 
The moment map for the  $\Gau(Q,\tau)$-action is  $\on{MCG}(\Sigma)$-invariant; hence we obtain a 
symplectic action of  $\on{MCG}(\Sigma)$ 
on the quotient. 
\end{proof}

The symplectic structure obtained in this way does not depend on the choice of $(P,\sigma)$ subject to the condition  
$\ez(P,\sigma)=\chi(\Sigma)$, since any two choices are related by a bundle isomorphism (Proposition 
\ref{prop:quantumnumber}). 
The intermediate space 
$\wh{\on{Teich}}(\Sigma)$ depends on the choice, but only through the boundary restriction $(Q,\tau)$. 
There is a canonical choice for the boundary restriction, and hence of the space $\wh{\on{Teich}}(\Sigma)$, 
coming from the theory of Drinfeld-Sokolov reduction. We will discuss it in the next 
section,  since it will also lead to a simpler description of $\on{Teich}(\Sigma)$.

\begin{remark}
It would be interesting to have a construction of the symplectic structure directly from the metric, as in the work of 
Tromba \cite{tro:tei} (see also Donaldson \cite{don:mom}, Diez-Ratiu \cite{die:gro}). 
\end{remark}

\section{$\on{Teich}(\Sigma)$ as a Hamiltonian Virasoro space} 
We will now verify that the map $\on{Teich}(\Sigma)\to \on{Proj}(\pS)$, taking the equivalence class of a hyperbolic structure on a surface with boundary to the induced projective structure on the boundary, is an affine moment map. 
The affine structure on 
$\on{Proj}(\pS)$ comes from its identification  with an affine subspace of the dual of the Virasoro Lie algebra $\mf{vir}(\pS)$, at a suitable non-zero level. We will give explicit formulas for the Hill operator on the boundary, in terms of data coming from the hyperbolic 0-metric.  

\subsection{Review of Hill operators and Virasoro algebra}\label{subsec:hill}
We shall need some background material. For more detailed information, see the standard references \cite{khe:inf,ovs:pro} as well as 
our earlier paper \cite{al:coad}.  Let ${\CC}$ be a compact, oriented 1-manifold. For $r\in \R$, we denote by $|\Omega|_\CC^r$ the space of $r$-densities. A $k$-th order differential operator $D\colon |\Omega|_\CC^{r_1}\to  |\Omega|_\CC^{r_2}$ has a principal symbol 
$\sigma_k(D)\in |\Omega|_\CC^{r_2-r_1-k}$. The principal symbol is scalar exactly when $r_2=r_1+k$. 
If $r_1+r_2=1$, the (formal) adjoint operator acts between the same spaces, and it makes sense to ask that $D$ be self-adjoint. A \emph{Hill operator} is a second order differential operator 
\[ L\colon |\Omega|_\CC^{-\f{1}{2}}\to  |\Omega|_\CC^{\f{3}{2}}\]
satisfying $L^*=L$ and $\sigma_2(L)=1$. The space of all Hill operators is an affine space
$\on{Hill}(\CC)$, with the space of quadratic differentials $|\Omega|^2_\CC$ as its space of translations. 
%
%
There is a  $\Diff_\oz(\CC)$-equivariant isomorphism 
\begin{equation}\label{eq:hillproj}
\on{Hill}(\CC)\stackrel{\cong}{\lra}\on{Proj}(\CC),
\end{equation}
taking  a Hill operator $L$ to the projective structure with charts  $(u_1:u_2)\colon U\to \RP(1)$, for 
local solutions $u_1,u_2\in |\Omega|_U^{-1/2}$ of $Lu=0$, with Wronskian $W(u_1,u_2)=-1$. The natural action of $\on{Diff}_+(\CC)$ on $\on{Hill}(\CC)$  is an affine action, with underlying linear action the coadjoint action. Here   $|\Omega|^2_\CC$ is seen as the  (smooth) dual to the space of vector fields $\on{Vect}(\CC)=|\Omega|_\CC^{-1}$. 
The Virasoro algebra $\mf{vir}(\CC)$ is the central extension of $\Vect(\CC)$ defined by this action (see Appendix \ref{app:A1}). The action of $\Diff_\oz(\CC)$ on $\mf{vir}^*_1(\CC)$ is the coadjoint Virasoro action; see e.g. \cite{dai:coa,kir:orb,laz:nor,seg:geo,wit:coa}.

Using a local coordinate $x$ on $\CC$, the $r$-density bundles are trivialized by the sections $|dx|^r$. In terms of this trivialization,  a Hill operator takes on the form 
\begin{equation}\label{eq:L} L=\f{d^2}{d x^2}+T(x)\end{equation}
for a \emph{Hill potential} $T$. For $\FF\in \on{Diff}(\CC)$, the Hill operator $\FF^{-1}\cdot L$ has 
Hill potential $\FF^{-1}\cdot T$ given by the formula 
 \begin{equation}\label{eq:hilltransformation}
(\FF^{-1}\cdot T)(x)=\FF'(x)^2 T(\FF(x))+\hh \S(\FF)(x)\end{equation}
with the Schwarzian derivative \cite{ovs:pro,ovs:what}
$\S(\FF)=\f{\FF'''}{\FF'}-\f{3}{2}\left(\f{\FF''}{\FF'}\right)^2$.  
The map \eqref{eq:hillproj} factors through the \emph{Drinfeld-Sokolov embedding} 
\begin{equation}\label{eq:dsembedding} \on{Hill}(\CC)\to \A^{\on{pos}}(Q)\end{equation}
for a canonically defined pair $(Q,\tau)$. 
The following coordinate-free description is due to Segal \cite{seg:geo}. Let $|\Lambda|^{-1/2}_\CC$ be the bundle of $-\hh$ densities. A
Hill operator $L$ determines a linear connection on the 1-jet bundle 
\begin{equation}\label{eq:e} E=J^1(|\Lambda|^{-1/2}_\CC)\end{equation}
with the property that $Lu=0$ if and only if $\nabla j^1(u)=0$. 
(By standard ODE theory, every solution is uniquely determined by its 1-jet at any given point.) 
Dually,  we obtain a connection on $E^*$. Dualizing the projection  $E\to |\Lambda|^{-1/2}_\CC$, we obtain a rank 1 subbundle of 
$E^*$, or equivalently a section of its projectivization. We take $Q\to \CC$ be the associated principal $G$-bundle,   
thus $\PP(E^*)=Q\times_G \RP(1)$, and  let $\tau\colon Q\to \RP(1)\cong \p\DD$ be the map defining this section. 
The connection on $E^*$ defined by a Hill operator descends to a $\tau$-positive connection on $Q$, defining the inclusion \eqref{eq:dsembedding}. The image of the Drinfeld-Sokolov embedding will be called the Drinfeld-Sokolov slice, denoted \begin{equation}\label{eq:dsslice}\ca{Z}\subset \A^{\on{pos}}(Q).\end{equation}
The bundle $V_\tau=\tau^*T\p\DD/G$ for Segal's $(Q,\tau)$ is canonically isomorphic to the tangent bundle
\[ V_\tau\cong T\CC.\]
Hence, given a connection $\vartheta\in \A(Q)$, the map
$\az\colon T\CC\to V_\tau$ from \eqref{eq:positivitycondition2} 
is scalar multiplication by a function $a$, and $\vartheta$ is positive if and only if $a>0$ everywhere. We have $\ca{Z}\subset a^{-1}(1)$.

The choice of  a local coordinate $x$ on $\CC$ determines a trivialization of $|\Lambda|_C^{-1/2}$, hence also of its jet bundle and consequently of $Q$. In this trivialization, $\tau(m,g)=g^{-1}\cdot 0$, and the Drinfeld-Sokolov embedding is given by the 
formula\footnote{In \cite{al:coad}, we worked with the bundle $E$ instead of $E^*$. The expression in \eqref{eq:drinfeldsolkolovnormalform} is therefore minus the transpose of  that used in \cite{al:coad}.}
\begin{equation}\label{eq:drinfeldsolkolovnormalform}
 T\,|\d x|^2\mapsto 
\left(\begin{array}{cc}
0 & 1\\ -T & 0
\end{array}
\right)\ \d x.\end{equation}
%
From the coordinate-free description, it is clear that $\Diff_+(\CC)$ acts on $Q$ by automorphisms preserving $\tau$, hence 
defining a splitting 
\begin{equation}\label{eq:dssplitting} \on{Diff}_+(\CC)\to \Aut_+(Q,\tau).\end{equation} The Drinfeld-Sokolov embedding is equivariant 
for this action. 
Using local coordinates to trivialize the bundles, this is given by 
$\FF^{-1}\mapsto (h,\FF^{-1})\in C^\infty(\CC,B^-)\rtimes \on{Diff}_+(\CC)$ with 	
\begin{equation}\label{eq:h}
 h=
 \left[\begin{array}{cc} 1 &0\\ \f{1}{2}\FF'' (\FF')^{-1} & 1
 \end{array}\right]\ 
 \left[\begin{array}{cc} (\FF')^{-\f{1}{2}} &0\\ 0 & (\FF')^{\f{1}{2}}
\end{array}\right].\end{equation}
On may check directly that $h$ is the unique $B^-$-valued function such that $h\bullet \FF^*A$
is again in the Drinfeld-Sokolov slice, with  
$T$ replaced by $\FF^{-1}\cdot T$.

\subsection{Hill potential in terms of adapted coframes}
Given a hyperbolic 0-metric $\gz$ on an oriented surface $\Sigma$ with boundary, we are interested in a description of the corresponding  Hill operator in terms of adapted coordinates $x,y$. Let $\alpha_1,\alpha_2$ be an adapted orthonormal coframe for $\gz$, with associated spin connection $\kappa$. Write 
\[ \alpha_1=\f{1}{y}(a(x)\d x+\ldots),\ \ 
\alpha_2=\f{\d y}{y}+s(x)\d x+\ldots,\ \   
\hh (\alpha_1+\kappa)=y (u(x)\d x+\ldots)\]
where the dots indicate regular 1-forms whose pullback to the boundary vanishes.  

\begin{proposition}\label{prop:Tformula}
The Hill potential corresponding to the hyperbolic 0-metric $\gz$ is given by the formula  
 	\begin{equation}\label{eq:T} T=\hh\left(\f{a''}{a}-\f{3}{2}\Big(\f{a'}{a}\Big)^2\right)-au-\f{1}{4}s^2-\hh \f{a'}{a}s+\hh s'.\end{equation}
 	
\end{proposition}
\begin{proof}
The 0-connection 1-form \eqref{eq:A} reads 
\[ \left(\begin{array}{cc}  \f{1}{2y}\ \d y +\hh s(x)\d x+\ldots&  \f{1}{y}   (a(x)\d x+\ldots)\\  y(u(x)\d x+\ldots)& 
-\f{1}{2y}\ \d y -\hh s(x)\d x+\ldots
\end{array}\right).\]
In Section \ref{subsec:relation}, we explained that the \emph{regular} connection 1-form, describing $\partial\theta\in \A^{\on{pos}}(Q)$, is obtained by applying the `singular gauge transformation' by $\on{diag}(y^{\hh},y^{-\hh})$, and 
pulling back to $y=0$. The result is the connection 1-form 
\eqref{eq:connectiononboundary} from the proof of Lemma \ref{lem:ds}. By working out the gauge transformation 
indicated there, taking the connection to Drinfeld-Sokolov normal form, one obtains $T$ as minus the lower left corner. 
The result of this straightforward calculation is \eqref{eq:T}. 
\end{proof}

\begin{example} 
The Hill potential 	for the Poincar\'{e} disk, with coordinates $\phi$ as in Example \ref{ex:example2}\ref{it:b}, is $T(x)=\f{1}{4}$. 
For the trumpet, with geodesic neck of length $\ell$ (Example \ref{ex:example2}\ref{it:c}), we obtain $T(x)=-\f{1}{4}\ell^2$. For the Fefferman-Graham coframe (Example \ref{ex:example2}\ref{it:d}), the Hill potential agrees with the function $T$ given in that formula.

\end{example}

\subsection{Hill potential in terms of geodesic curvature}
We will now give a second description of the Hill operator of a hyperbolic 0-metric $\gz$,  motivated by the discussion in  Maldacena-Stanford-Yang \cite[Section 3]{mal:con}. Observe that the function $a(x)$ 
in \eqref{eq:T} may be read off from the leading term of the volume form;
\[ \d\vol_{\gz}=\f{1}{y^2} \big(a(x)+O(y^1)\big) \d x\wedge \d y.\]

For $y>0$, let $k(x,y)$ be the geodesic curvature of the curve $t\mapsto (x+t,y)$. 
Recall that for the standard hyperbolic metric on the upper half plane, the horizontal lines all have geodesic curvature equal to $1$. Hence
	$k(x,y)=1$ for all $x,y\in \HH$. 	
It turns out that in general, $k(x,y)=1+O(y^2)$: 

\begin{lemma}\label{lem:limit}
	For every hyperbolic 0-metric, the limit 
	\[ c(x)=\lim_{y\to 0} \f{k(x,y)-1}{y^2}\]	
	exists and defines a smooth function of $x$. 
\end{lemma}
\begin{proof}
	By Theorem \ref{th:normalform}, the hyperbolic 0-metric may be written 
	$\gz=\f{1}{g^2}(\d f^2+\d g^2)$
	for functions $f,g$ with $\f{\p f}{\p x}(x,0)>0$ and $g(x,0)=0,\ \f{\p g}{\p y}(x,0)>0$. 
	That is, $(f,g)$ defines a local isometry to $\ol{\HH}$. 
	The image of the curve $t\mapsto (x+t,y)$ under this isometry is the curve $t\mapsto (f(x+t,y),g(x+t,y))$ in $\HH$; its geodesic curvature $k(x,y)$ is computed as 
	\[ k=\frac{f'}{\big((f')^2+(g')^2\big)^{1/2}}+f \f{f' g''-f'' g'}{\big((f')^2+(g')^2\big)^{3/2}},\]
	where the prime denotes $x$-derivatives. Substituting Taylor series 
	\[ f(x,y)\sim \sum_i f_i(x)y^i,\ 
	g(x,y)\sim \sum_i g_i(x) y^i\] 
	one finds, by direct but somewhat lengthy calculation, 
	\[ k(x,y)=1+c(x) y^2+y^3\]
	with 
	\begin{equation}\label{eq:kx}
	c=\f{g_1 g_1''}{(f_0')^2}-\f{g_1 g_1' f_0''}{(f_0')^3}-\hh \f{(g_1')^2}{(f_0')^2}.
	\end{equation}
(The calculation requires writing $f,g$ up to second order, but $f_1,f_2,g_2$ do not enter the final expression.) 
\end{proof}

\begin{theorem}\label{th:hill2}
	The Hill potential is given by 
	\begin{equation}\label{eq:tformula} T=\hh \left(\f{a''}{a}-\f{3}{2}\big(\f{a'}{a}\big)^2\right
	) +\f{a^2}{2} c\end{equation}
	where $c$ is obtained from the limit of the geodesic curvatures of the curves $t\mapsto (x+t,y)$ as 
	$c(x)=\lim_{y\to 0}(k(x,y)-1)/y^2$. 
\end{theorem}
\begin{proof}
	Continuing the notation from the proof of Lemma \ref{lem:limit},  we may take 
	$\alpha_1=\f{\d f}{g},\ \alpha_2=\f{\d g}{g}$
	as an adapted orthonormal coframe. Using the Taylor expansion of $f,g$, we find 
	\[ \alpha_1=\f{1}{y} (\f{f_0'}{g_1}\d x+\ldots),\ \ \ \alpha_2=\f{\d y}{y}+\f{g_1'}{g_1}\d x+\ldots,\ \ \alpha_1+\kappa=0\]
	where dots indicate terms that pull back to zero on the boundary $y=0$. Hence, the functions $a,s,u$ are given by 
	\[ a=\f{f_0'}{g_1},\ \ s=\f{g_1'}{g_1},\ \ u=0.\]
	Using the formula \eqref{eq:kx} for $k(x)$, this gives 
	\[ c=\f{1}{a^2} \Big(s'-\f{a'}{a}s-\f{1}{2}s^2\Big).\]
	Now use \eqref{eq:T}.  
\end{proof}


\subsection{Verifying the moment map condition}
We are now in position to describe the moment map for the $\wt{\on{Diff}}_\oz(\pS)$-action on the infinite-dimensional Teichm\"uller space. 
\begin{theorem}\label{th:momentmap}
The action of $\wt{\on{Diff}}_\oz(\pS)$ on $\on{Teich}(\Sigma)$ is Hamiltonian, with moment map 
\[ \Phi\colon \on{Teich}(\Sigma)\to \mf{vir}^*_{-1}(\p\Sigma),\ \ [\gz]\mapsto -L\]
taking the equivalence class of a hyperbolic structure to minus the Hill operator for the associated projective structure on the boundary. 
\end{theorem}

\begin{remark}
We obtain a moment map at level $-1$ due to our specific choice of metric $\xi\cdot\eta=\on{tr}(\xi\eta)$ on $\g$. Multiplying the metric by a nonzero factor,  the symplectic form and 
moment map (and in particular its level) scale accordingly. 	
\end{remark}

Our starting point is the description (Theorem \ref{th:teichhat})
\[
 \on{Teich}(\Sigma)=(\wh{\on{Teich}}(\Sigma)\times \O^-)\qu \Gau(Q,\tau),\]
 where $(Q,\tau)$  is the boundary restriction of $(P,\sigma)$. 
 We shall take this 
 boundary restriction to be Segal's bundle from Section \ref{subsec:hill}. Denote $\CC=\pS$.

The action of $\wt{\Diff}_\oz(\CC)$ is obtained as a quotient of the action of (a cover of) $\Aut_\oz(Q,\tau)$  
on both spaces,  $\wh{\on{Teich}}(\Sigma)$ and $\O$. Recall that for Segal's bundle, there is a canonical splitting  $\Diff_\oz(\CC)\to \Aut_\oz(Q,\tau)$.  This lifts to the universal  covering. Hence, we may compute the $\Diff_\oz(\CC)$-part of the  moment map on both spaces.

The choice of a  a coordinate  $x$ on the boundary gives a trivialization $Q=\CC\times G$. In terms of this trivialization, the connection $\partial\theta$ is described by a connection 1-form $A$ as in \eqref{eq:connectiononboundary}. 

\begin{proposition}
The moment map for the $\wt{\Diff}_\oz(\pS)$-action on $\wh{\on{Teich}}(\Sigma)$ is given in coordinates by 
\begin{equation} \label{eq:hatmoment}
\wh{\on{Teich}}(\Sigma)\to |\Omega|^2_{\pS},\ \ 
[\theta]\mapsto \big(-\f{1}{2} s' + \f{1}{4} s^2+au-\f{1}{2} a'')\ |\d x|^2.\end{equation}
Here the functions $a,u,s$ are defined by \eqref{eq:connectiononboundary}. 
\end{proposition}
\begin{proof}
The boundary term of the moment map is given by $[\theta]\mapsto (A,\hh \on{tr}(A^2))\in \Omega^1(\CC,\g)\times |\Omega|^2_\CC$; see Appendix \ref{app:AB}. 
On the other hand, the coordinate expression of the inclusion 
$\Vect(\CC)\to \gau(Q,\tau)\rtimes \Vect{\CC}$ 
is given  by the infinitesimal version of \eqref{eq:h}\footnote{With our sign conventions, the flow of the vector field 
	$ f\,\f{\p}{\p x}$ is of the form $F_t(x)=x-t f(x)+O(t^2)$. An additional sign arises comes from the choice of identification $\gau(Q)\cong C^\infty(\CC,\g)$ used in appendix \ref{app:AB}.}
\[ f\,\f{\p}{\p x}\mapsto 
\left(\left(\begin{array}{cc} \hh f' &0\\ -\f{1}{2}f''  & -\f{1}{2} f'
\end{array}\right),\ f\,\f{\p}{\p x}\right).\]
Using the expression \eqref{eq:connectiononboundary} for $A$, the corresponding component of the moment map is 
\[ \int_{S^1} \on{tr}\left( \left(\begin{array}{cc}
\hh s & a\\ u & -\hh s
\end{array}
\right) 
\,\left(\begin{array}{cc}\hh f' & 0\\ -\hh f'' & -\hh f'
\end{array}\right)\right)\ +\hh \int_{S^1} \on{tr}\left(\begin{array}{cc}
\hh s & a\\ u & -\hh s
\end{array}
\right)^2 f\]
\[ =\int_{S^1}\left(\hh sf'-\hh af''+\f{1}{4}sf+auf
\right)
=\int_{S^1} \left(- \hh s'-\hh a''+\f{1}{4} s^2+au
\right)f. \qedhere \]
\end{proof}

We recognize some, but not all, of the terms in the formula \eqref{eq:T} for the Hill potential. One expects to obtain the 
remaining terms from a calculation of the $\wt{\Diff}_\oz(\CC)$-moment map on $\O$. Through explicit calculation, we checked  that this is indeed the case, thereby obtaining a proof of Theorem \ref{th:momentmap}. However, there is a much simpler argument, using the Drinfeld-Sokolov slice: 

\begin{proof}[Proof of Theorem \ref{th:momentmap}]
	Recall that the moment map $\Psi\colon \wh{\on{Teich}}(\Sigma)\to \A(Q,\tau)$ is given by $[\theta]\mapsto 
	\p\theta\mod \on{ann}\gau(Q,\tau)$, and the set of all $\p\theta\mod \on{ann}\gau(Q,\tau)$ is a single coadjoint orbit $\O=\A^{\on{pos}}(Q,\tau)$.
	The Drinfeld-Sokolov slice $\ca{Z}\subset \A^{\on{pos}}(Q)$ 
	descends to a slice for the $\Gau(Q,\tau)$-action on $\O=\A^{\on{pos}}(Q)/\on{ann}(Q,\tau)$, consisting of just a single point, 
	$\mu_0\in \O$, and the stabilizer of this point under $\Gau(Q,\tau)$ is \emph{trivial}. (See Lemma \ref{lem:ds}.) 
	Letting 
	$\mu_0$ be the corresponding point in $\O$, we have 
	\[ \on{Teich}(\Sigma)=\Psi^{-1}(\mu_0)\subset \wh{\on{Teich}}(\Sigma)\]
	as a symplectic submanifold. The moment map for the $\wt{\Diff}_\oz(\CC)$-action on $\on{Teich}(\Sigma)$ may be computed by restricting the moment map to this cross-section. 
	
	Using coordinates, as above,  $\ca{Z}$ is given by  $a=1,s=0,u=-T$ where $T$ is the Hill potential. (The point 
	$\mu_0\in \O$ is the point given by $a=1,s=0$). Hence, on $\Psi^{-1}(\mu_0)$ the moment map restricts to 
	 $[\theta]\mapsto -T$.  
\end{proof}

\section{The symplectic form in Fenchel-Nielsen coordinates}\label{sec:fn}

\subsection{Fenchel-Nielsen parameters}\label{subsec:fn}
Let $\Sigma$ be a compact, connected, oriented surface (possibly with boundary), of negative Euler characteristic 
$\chi(\Sigma)<0$. 
The construction of Fenchel-Nielsen parameters on $\on{Teich}(\Sigma)$  
for surfaces without boundary is well-explained in 
\cite{far:pri}; we describe a straightforward generalization to the case of a possibly non-empty boundary. 

Recall first that every simple, closed curve $\mathsf{D}\subset \Sigma$, neither contractible nor homotopic to a boundary 
component, determines a \emph{twist flow} $\R\times \on{Teich}(\Sigma)\to \on{Teich}(\Sigma)$: 
Given $\gz$, one obtains a new metric $\gz_\tau$ by cutting the surface along the geodesic homotopic to $\mathsf{D}$, 
and gluing  the two sides back together after rotating (twisting) one of the ends by an amount $\tau$. 

\begin{remark}
A more detailed description:  Given $[\gz]\in \on{Teich}(\Sigma)$, choose a representative $\gz$ 
having $\mathsf{D}$ as a closed geodesic.  A collar neighborhood $U$ of $\mathsf{D}$ is isometric to a neighborhood of the geodesic of a hyperbolic cylinder (see \ref{subsubsec:cylinder}), and so is isometric to $\mathsf{D}\times (-\epsilon,\epsilon)$ with the hyperbolic metric \eqref{eq:doubletrumpet}. For any $\tau\in\R$, we obtain a new hyperbolic metric $\gz_\tau$ by letting $\gz_\tau|_{\Sigma-U}=\gz|_{\Sigma-U}$ and taking 
$\gz_\tau|_U$ to be the pullback of $\gz|_U$ under the  diffeomorphism 
\begin{equation}\label{eq:twistflow} (x,u)\mapsto \big(x+\f{\tau}{\ell} f(u),u\big),\end{equation} 
where $f(u)=0$ for $u<-\f{1}{2}\epsilon$ and $f(u)=1$ for $u>\hh \epsilon$. 
The twist flow is given by $[\gz]\mapsto [\gz_\tau]$. 
\end{remark}

Since $\chi(\Sigma)<0$, we may choose a pairs-of-pants decomposition of $\Sigma$. There are $2g-2+r=-\chi(\Sigma)$ distinct pants; their boundary curves consist of the boundary loops 
$\CC_j,\ j=1,\ldots,r$ of $\Sigma$ and $3g-3+r$ simple closed curves $\mathsf{D}_i\subset \on{int}(\Sigma)$. 

Each of the $\mathsf{D}_i$ defines a length parameter $\ell_i>0$ (the length of the unique closed geodesic homotopic to $\mathsf{D}_i$), as well as a twist flow. 
In addition, each boundary component $\CC_j$ determines a length parameter $\ell_j$ (given by the length of the unique geodesic 
of $[\gz]$ homotopic to the ideal boundary $\CC_j$) as well as an action of $\wt{\Diff}_\oz(\CC_j)$ (coming from the action of diffeomorphisms in $\bb\Diff_\oz(\Sigma)$ that are supported in collar neighborhoods of the $\CC_j$). These actions 
on $\on{Teich}(\Sigma)$ all commute, with quotient $\R_{>0}^{3g-3}\times \R_{>0}^r$ given by the length parameters.   
This action has a global slice, determined by the choice of a system of \emph{model seams}. Choose  an embedded 1-dimensional submanifold  $E\subset  \Sigma$  
with boundary $\p E\subset \p\Sigma$, in such a way for any two distinct boundary circles of a given pants $P$, 
there is a unique component of $P\cap E$ connecting those two boundary components. We also assume that $P\cap E$ meets these 
boundary components transversely. The three components of $P\cap E$ are the model seams for the pair of pants $P$. 
Finally, choose orientation preserving parametrizations $\mathsf{C}_j\cong S^1$, such that 
$\p E\cap \mathsf{C}_j$ maps to the antipodal points $\{-1,1\}\in S^1\subset \C$.   

 The desired slice consists of all $[\gz]\in \on{Teich}(\Sigma)$, where $\gz$ is a hyperbolic 0-metric such that  
 (i) all $\mathsf{D}_i$ are geodesics, (ii) the connected components of $E-\p E\subset \on{int}(\Sigma)$ are geodesics, (iii) the projective structure on the boundary components $\CC_j$ is constant (i.e., $S^1$-equivariant). This gives 
 an identification 
\begin{equation}\label{eq:FN1}
\on{Teich}(\Sigma)\cong (\R_{>0}\times \R)^{3g-3+r}\times \prod_{j=1}^r (\R_{>0}\times \wt{\Diff}_\oz(S^1))
\end{equation}
with the slice given as $(\R_{>0}\times 0)\times \prod_{j=1}^r (\R_{>0}\times \on{Id})^r$. Denote the corresponding 
parameters by $\ell_i,\tau_i,\ell_j,\FF_j$; we choose the parametrization in such a way that the $i$-th twist flow is given by 
$\tau_i\mapsto \tau_i+\tau$ (leaving all other parameters unchanged) and the $j$-th action of $\FF\in \wt{\Diff}_\oz(S^1)$ 
is given by $\FF_j\mapsto \FF_j\circ \FF^{-1}$ (leaving all other parameters unchanged).

\subsection{Related Teichm\"uller spaces}
Given a hyperbolic 0-metric $\gz$ on $\Sigma$, each boundary component $\CC_j$ determines  a unique simple, closed geodesic $\CC_j'\subset \Sigma$ homotopic to $\CC_j$; this is the geodesic end of the $j$-th boundary trumpet.  Removing the trumpets creates a surface 
$\Sigma'$ with geodesic boundary $\sqcup_{j}\CC_j'$, called the \emph{compact core} of $\Sigma$. Of course, $\Sigma'$ is diffeomorphic to $\Sigma$
(as a surface with boundary). 
The map \begin{equation}\label{eq:principalbundle}
\on{Teich}(\Sigma)\lra  \on{Teich}_{\on{geod}}(\Sigma),\ \ \ 
\end{equation}
taking the equivalence class of a hyperbolic 0-metric on $\Sigma$ to the equivalence class of the (ordinary) hyperbolic metric on $\Sigma$ with \emph{geodesic boundary}, is the quotient maps for the action of $\prod_{i=1}^r \wt{\Diff}_\oz(\CC_i)$. The corresponding Fenchel-Nielsen description just omits the $\wt{\Diff}_\oz(S^1)$-factors in \eqref{eq:FN1}. Fixing the lengths $b_j$ of the boundary components, one obtains the space $\on{Teich}_{\on{geod}}(\Sigma,b_1,\ldots,b_r)$.
As another variation, having chosen parametrizations $\CC_j\cong S^1$ of the boundary components, 
we may consider the subspace 
\[ \on{Teich}_{\on{bordered}}(\Sigma)\subset \on{Teich}(\Sigma)\] 
for 
which the projective structure on each $\CC_j$ is `constant' (i.e., invariant under rigid rotations). On this subspace, we have a residual action of $\R^r$, where the $j$-th copy of $\R$ rotates the $j$-th boundary component. This version of the Teichm\"uller space may be interpreted as a space of hyperbolic metrics with geodesic boundary, together with a `marking' on each boundary component. 
This space of `bordered' hyperbolic metrics  appears in Mirzakhani's work, see \cite[Section 4]{mir:wei}.
	The Fenchel-Nielsen description becomes 
	\[  \on{Teich}_{\on{bordered}}(\Sigma)\cong (\R_{>0}\times \R)^{3g-3+2r}.\] 
	One can also consider mixtures of such spaces, e.g., taking some boundary components to be ideal boundaries (with $\varrho^{-2}$-boundary behaviour of the metric), other boundaries as  marked
	geodesic boundaries.

\subsection{Teichm\"uller space of the trumpet}
Let $\wt{N}$ denote the Teichm\"uller space of hyperbolic 0-metrics on $ S^1\times [-\infty,\infty] $, such that the induced projective structure on the left boundary $S^1\times \{-\infty\}$ is constant, and denote by $N$ the corresponding Riemann moduli space.  
As explained above, $N$ may also be regarded as a moduli space of hyperbolic metrics on $S^1\times [0,\infty)$ for which $S^1\times \{0\}$ is a geodesic, of some length $\ell>0$, while $\gz$ has the boundary behaviour of a 0-metric along the ideal boundary $S^1\times \{\infty\}$. 
As a space, 
\begin{equation}\label{eq:trumpet} 
 N=\R_{>0}\times \Diff_\oz(S^1)\end{equation}
where the $\R_{>0}$ factor indicates the length $\ell$ of the geodesic boundary. 
This space comes with an action of 
$\Diff_\oz(S^1)$ by $\FF_1\cdot (\ell,\FF)=(\ell,\FF\circ \FF_1^{-1})$ 
and an action of $S^1=\R/\Z$ by $t\cdot  (\ell,\FF)\mapsto (\ell,\FF+t)$. (Here multiplication on $S^1=\R/\Z$ is written additively.) 
The following result describes the symplectic structure on $N$; the 2-form on $\wt{N}$ is obtained by pullback. (For a more conceptual explanation of the formula, see Appendix \ref{app:groupoid}.) 

For $\ell>0$, we have the Hill operator $L(\ell)=\f{d^2}{d x^2}+T(\ell)$ with the constant Hill potential 
$T(\ell)=-\f{1}{4}\ell^2\in \on{Hill}(S^1)$.

\begin{theorem}[Trumpet] \label{th:trumpet}
	The space \eqref{eq:trumpet} has a unique invariant symplectic form $\omega_N$, in such a way that the 
	$\Diff_\oz(S^1)$ is Hamiltonian, with moment map $(\ell,\FF)\mapsto -\FF^{-1}\cdot L(\ell)\in \on{Hill}(S^1)$. This 2-form is given by 
	the formula 
	\begin{equation}\label{eq:omegan}
	\omega_N=-\f{1}{4}\d \int_{S^1}\big(\ell^2\FF'\,\d\FF +(\FF')^{-1}\d\FF''
	\big)
	\end{equation}
	The $S^1$-action is Hamiltonian as well, with moment map $(\ell,\FF)\mapsto \f{1}{4}\ell^2$. The symplectic quotient at $\f{1}{4}\ell^2$ for the latter action is the coadjoint Virasoro orbit through $L(\ell)$ (with the opposite symplectic structure). 
\end{theorem}
Before proving this result, we have to explain the ingredients of \eqref{eq:omegan}. For fixed $x\in S^1$, we have the evaluation map $\on{ev}_x\colon \Diff(S^1)\to \R/\Z,\ \FF\mapsto\FF(x)$. As in \cite{al:coad} we shall denote this function on $ \Diff(S^1)$ simply by $\FF(x)$ (thinking of $\FF$ as a variable). 
The exterior derivative $\d(\FF(x))$ of this function is a 1-form on $\Diff(S^1)$; letting $x$ vary this is a 1-form on diffeomorphisms with values in periodic functions, 
 \[ \d\mathsf{F}\in \Omega^1(\Diff_\oz(S^1),|\Omega|^0_{S^1}).\]
On the other hand, for fixed $\mathsf{F}$ we may take the exterior derivative of the function $x\mapsto \mathsf{F}(x)$. We shall denote it by 
\[ \FF'\in  \Omega^0(\Diff_\oz(S^1),|\Omega|^1_{S^1})\]
(a more accurate notation would be $\FF'(x)|\d x|$).
Higher derivatives are defined as well; for example, $\FF''$ is naturally a function on $\Diff_\oz(S^1)$ with values in quadratic differentials. Since $\FF'(x)>0$ everywhere, we may also consider $1/\FF'\in  \Omega^0(\Diff_\oz(S^1),|\Omega|^{-1}_{S^1})$. 
With this understanding, each of the terms in \eqref{eq:omegan} is a 2-form on $\Diff_\oz(S^1)$ with values in 
$|\Omega|^{1}_{S^1}$; integration of the 1-density over $S^1$ results in a 2-form on $\Diff_\oz(S^1)$.

\begin{proof}[Proof of Theorem \ref{th:trumpet}]
Expanding \eqref{eq:omegan}, we have 
\begin{equation}\label{eq:omegan2}
	\omega_N=\f{1}{4} \int_{S^1}\ 	\ \Big( -\FF' \d \ell^2\wedge \d \mathsf{F}-\ell^2\,\d \FF'\wedge \d \FF
+\f{\d\FF'\wedge\d\FF''}{(\FF')^2}
\Big).\end{equation}
To check that  it does satisfies the moment map condition, we consider its contraction with a left-invariant vector field $v^L$ on $\Diff_\oz(S^1)$ 
corresponding to $v\in \Vect(S^1)$. The flow of $v^L$ on $\Diff_\oz(S^1)$ is given in terms of the flow 
$t\mapsto \exp(tv)$ by $\FF\mapsto \FF\circ \exp(-tv)$. As explained in \cite[Lemma 4.8]{al:coad}, if $v=f(x)\partial_x$ then
\[ \iota(v^L)\d\FF=-\FF'\, f.\]
The contractions with $\d\FF',\ \d\FF''$ are obtained by taking derivatives of this expression. Hence, 
\[ \iota(v^L)\omega_N=\f{1}{4} \int_{S^1} 
\Big(-(\FF')^2 f \d \ell^2-\ell^2 f \FF'\d\FF' -\ell^2(-\FF'f)'\d\FF
+\f{ (-\FF' f)' \d\FF''-(-\FF' f)'' \d\FF'}{(\FF')^2}       \Big) \]  
Use integration by parts so that no derivatives of $f$ appear:
\[  \iota(v^L)\omega_N=\f{1}{4} \int_{S^1} 
\Big(-(\FF')^2 \d \ell^2-2\ell^2 \FF'\d\FF'
+\FF'\Big( \f{\d \FF''}{(\FF')^2}
\Big)'+\FF'\Big( \f{\d \FF'}{(\FF')^2}\Big)''
\Big)f. 
\]
After simplifications, this becomes  
\[ \iota(v^L)\omega_N=\d\int_{S^1} \Big(\big(-\f{1}{4} (F')^2 \d\ell^2 +\f{1}{2} \S(\FF)\big)  \Big) f
=\d\int_{S^1} \big(\FF^{-1}\cdot T(\ell)\big)\,f
.\]
where $T(\ell)=-\f{1}{4}\ell^2$ is the Hill potential corresponding to $\ell$. 	This shows that $(\ell,\FF)\mapsto -\FF^{-1}\cdot T(\ell)$ is a moment map for the action. Consider on the other hand the $S^1$-action $\FF\mapsto \FF+t\mod\Z$. Letting $Z$ denote its generating vector field, we have $\iota_Z \d \FF=1$, hence $\iota_Z\d \FF'=0,\ \iota_Z\d\FF''=0$. It follows that 
\[ \iota(Z)\omega_N=\f{1}{4}\int_{S^1} \FF' \d\ell^2=\f{1}{4}  \d\ell^2\]
where we used $\int_{S^1}\FF'=1$ by fundamental theorem of calculus. It follows that $(\ell,\FF)\mapsto -\f{1}{4}\ell^2$ is a moment map for this action. The reduction of $N$ with respect to this $S^1$-action, at level $-\f{1}{4}\ell^2$, 
is $\on{Diff}_+(S^1)/S^1$ with a closed $\Diff_\oz(S^1)$-invariant 2-form whose moment map gives a bijection 
onto $\Diff_\oz(S^1)\cdot L(\ell)\subset \mf{vir}^*_1(S^1)$. The reduction hence equals  
the (hyperbolic) coadjoint Virasoro orbit through $L(\ell)$. The fact that all the $S^1$-reduced spaces of $(N,\omega_N)$ 
are symplectic implies that $\omega_N$ must itself be symplectic.  The uniqueness part for $\omega_N$ follows since the difference of two 2-forms on $N$ satisfying the moment map condition is basic for the $\Diff_\oz(S^1)$-action, and hence is zero since the quotient is 1-dimensional. 
\end{proof}


\subsection{Fenchel-Nielsen description of the symplectic form}
For $i=1,\ldots,3g-3+r$, let  $\ell_i,\tau_i$ be the length and twist parameters with respect to $\mathsf{D}_i$, thought of as functions on $\on{Teich}(\Sigma)$. Also, for $j=1,\ldots,r$ let 
\[ \pi_j\colon \on{Teich}(\Sigma)\to N\]
be the map given by projection to the $j$-th boundary factor in \eqref{eq:FN1} (the Teichm\"uller space of the $j$-th trumpet), followed by the quotient map $\wt{N}\to N$. 

\begin{theorem}\label{th:fn}
In terms of Fenchel-Nielsen parameters \eqref{eq:FN1}, the symplectic form on $\on{Teich}(\Sigma)$  is given by 
\begin{equation}\label{eq:wolpertformula}
 \omega=	\hh \sum_{i=1}^{3g-3+r} \d \ell_i\wedge \d \tau_i+\sum_{j=1}^r \pi_j^*\omega_{N}.\end{equation}
\end{theorem}
In the case without boundary, this is the well-known  Wolpert formula \cite{wol:sym} for the Weil-Petersson symplectic form. 
The first part of the following argument is adapted from 
\cite[Section 3.3.2]{saa:jt}. 

\begin{proof} 
	We verify the formula at any given $[\gz]\in \on{Teich}(\Sigma)$. Pick a representative $\gz$ as in Section \ref{subsec:fn}; in particular, the $\mathsf{D}_i$ are geodesics. 
    It suffices to verify the formula on tangent vectors of the following types: infinitesimal changes of length or twist parameters for the curves $\mathsf{D}_i$,  tangent vectors 
	$v\in \Vect(\CC_j)$ corresponding to the $\wt{\Diff}_\oz(\CC_j)$-factors, as well as infinitesimal changes of the length parameters for $\CC_j'$, the geodesic ends of the trumpets. 
	These tangent vectors are realized by variations of the hyperbolic metric $\gz$. 
	
Consider a fixed $\mathsf{D}=\mathsf{D}_i$. 
We may introduce coordinates on some collar neighborhood of $\mathsf{D}$ so that $\gz$ is given by 
	the metric of the hyperbolic cylinder, \eqref{eq:doubletrumpet}, with the coframe (Example \ref{ex:example2}\ref{it:c}) 
	\[ \alpha_1=\cosh(u)\ell\ \d x ,\ \   \alpha_2=-\d u,\ \ 
	\kappa=-\sinh(u)\ell\ \d x\]
	and corresponding connection one-form
	\[ 
	A=\hh \left(\begin{array}{cc}  -\d u &  e^{u}\ell\d x\\ e^{-u}  \ell \d x & \d u\end{array}\right).
	\]
	Recall now the description of Fenchel-Nielsen flow, using pullback under \eqref{eq:twistflow}. 
	The corresponding $A_\tau$ is obtained by pullback: 
	\[ 
	A_\tau=\hh \left(\begin{array}{cc}  -\d u &  e^{u}(\ell \d x+\tau f'(u)\ \d u)\\   e^{-u}(\ell \d x+\tau f'(u)\d u) & \d u\end{array}\right).
	\] 
	Note that $A_\tau$ agrees with $A$ for $|u|\ge \epsilon$, hence it defines a new global connection $\theta_\tau$.
	Replacing $\tau$ with $\tau_t$, and taking a $t$-derivative, this gives the tangent vector 
	\[ b=\hh \left(\begin{array}{cc}  0&  e^{u}\\   e^{-u}& 0\end{array}\right) \dot{\tau}_0 f'(u)\,\d u.\]
		
	The tangent vector corresponding to a change of the length parameter $\ell$ is obtained by replacing $\ell$ with $\ell_t$ and taking a $t$-derivative: 
	\[ a=\hh \left(\begin{array}{cc}  0&  e^{u}\\   e^{-u}& 0\end{array}\right) \dot{\ell}_0\ \d x.\]
	We may arrange that $a$ has this form on the collar neighborhood $|u|\le \epsilon$, but vanishes outside of 
   a larger collar neighborhood (say, $|u|\le 2\epsilon$). 
	Hence, the Atiyah-Bott form on these tangent vectors evaluates to 
	\[ 
	\omega_{AB}(a,b)=\int_\Sigma \on{tr}(a b)=\f{1}{2} \int_{|r|\le \epsilon} \dot{l}_0\dot{\tau}_0 \d x\wedge f'(u) \d u=\f{1}{2}
	\dot{l}_0\dot{\tau}_0.\]
	This shows $\omega(\f{\p}{\p \ell_i},\f{\p}{\p \tau_i})=\f{1}{2}$. 
	On the other hand, if $a,b$ are tangent vectors corresponding to twist or length deformations for \emph{non-intersecting} geodesics, 
	then $\omega_{AB}(a,b)=0$ since we may take the support of $a,b$ to be disjoint. Thus, for example,  
	$\omega(\f{\p}{\p \ell_{i_1}},\f{\p}{\p \ell_{i_2}})=0$ for $i_1\neq i_2$. 
	Similarly, the pairing of the tangent vectors $\f{\p}{\p \ell_{i}},\ \f{\p}{\p \tau_{i}}$ with 
	a tangent vector $\f{\p}{\p \ell_j}$, corresponding to the change of length parameter for the $j$-th trumpet, is
	zero. 
	It remains to check \eqref{eq:wolpertformula} on pairs of tangent vectors, one of which is a 
	vector field $v\in \Vect(\CC_j)$ on the ideal boundary of the $j$-th trumpet. 
  	For this, it suffices to observe that both sides satisfy  the moment map condition 
	$\omega(v,\cdot)=\l \d L_j,v\r$ where $L_j$ is the Hill potential for the $j$th boundary. 
\end{proof}

\subsection{Darboux coordinates on the trumpet space}\label{subsec:trumpet}
The expression \eqref{eq:wolpertformula} for the symplectic form on $\on{Teich}(\Sigma)$ involves the symplectic structure $\omega_N$ on the space $N=\Diff_\oz(S^1)\times \R_{>0}$ associated to the trumpet end. We may go one step further and introduce Darboux coordinates on the space $N$, and hence on $\on{Teich}(\Sigma)$. 

Consider the symplectic structure on $\wt{N}=\R_{>0}\times \wt{\Diff}_\oz(S^1)$, given by \eqref{eq:omegan} 
with $\FF$ replaced by a 
lift $\wt{\FF}$ to the universal cover.  We may regard $\wt{\FF}$ as $\Z$-equivariant function on $\R$, that is,   $\wt{\FF}(x+1)=\wt{\FF}(x)+1$. 
The expression 
\begin{equation}\label{eq:u} u(x)= \log(\wt{\FF}'(x))+{\ell} (\wt{\FF}(x)-x);\end{equation}
is a periodic function on $\R$. Taking into account the dependence on $\ell,\wt{\FF}$, this is a function on 
$\wt{N}$ with values in $|\Omega|^0_{S^1}$. Thus $\d u\wedge \d u'$ is a 2-form on $N$ with values in $|\Omega|^1_{S^1}$; integrating over $S^1$ it is a 2-form on $\wt{N}$. 

\begin{proposition}
\[  \omega_{\wt{N}}=-\hh \d\ell\wedge \d u_0+\f{1}{4}\int_{S^1}\ \d u\wedge \d u'\]
where $u_0=\int_{S^1} \ u$ (a scalar function of $(\ell,\wt\FF)\in \wt{N}$) 
\end{proposition}
\begin{proof}
We work out the terms appearing in 
\[ \f{1}{4}\int_{S^1}\	\d u\wedge\d u'
=\f{1}{4}\int_{S^1} \d \big(\log(\wt{\FF}')+{\ell} (\wt{\FF}-\on{Id})\big)\wedge \d \big(
\log(\wt{\FF}')+{\ell} (\wt{\FF}-\on{Id})
\big)'\]
according to their homogeneity with respect to $\ell$, and compare to the corresponding terms in \eqref{eq:omegan2}. 
The term of homogeneity $0$ is $\f{1}{4}(\wt{F'})^{-2}\d\wt{F}'\wedge \d \wt{F}''$, matching that in \eqref{eq:omegan2}.
The terms of homogeneity $1$ 
 are (using integration by parts to combine two terms)  
\[ 
\hh\int_{S^1} \d \wt{F}'\wedge \d\ell+ \hh\int_{S^1}  \d\ell \wedge \d\log (\wt{F}')
\]
The first integral is zero, by the fundamental theorem of calculus, while the second integral gives one of the terms of $\hh\d\ell\wedge \d u_0$. The terms of homogeneity 2 are 
\begin{align*}
=\f{1}{4}\int_{S^1} \ell^2 \d\wt{\FF}\wedge \d\wt{\FF}'
 +\hh \int_{S^1} (\wt{\FF}'-1)\d\wt{\FF}\wedge \ell  \d\ell
\end{align*}
(again we used an integration by parts to combine two terms). The integral 
\[ -\int_{S^1}\d \wt{F}\wedge   \ell \d\ell=\d\ell\wedge \d \int_{S^1} \ell\wt{F}
=\d\ell\wedge \d \int_{S^1} \ell(\wt{F}-\on{Id})
\]
gives one of the terms in $ \d\ell\wedge \d u_0$. The remaining terms match the corresponding terms in \eqref{eq:omegan2}. 
\end{proof}

Having $\omega_{\wt{N}}$ in this form, it is straightforward to introduce Darboux coordinates, by  its Fourier expansion: 
\[ u(x)=\sum_{n\in \Z} u_n e^{2\pi i nx}.\]
We obtain:
\begin{proposition}\label{prop:darboux} The symplectic form on $\wt{N}$ is given by 
\[ \omega_{\wt{N}}=-\hh \d\ell\wedge\d u_0+\pi i\sum_{m>0} m\ \d u_{-m}\wedge \d u_m\]	
\end{proposition}
 
Rewriting the second sum in terms of real and imaginary parts of $u_m$, one obtains a Darboux normal form.

\begin{appendix}

\section{Affine moment maps, central extensions}
\subsection{Central extensions}\label{app:A1}
Let $H$ be a Lie group with Lie algebra $\h$. Given a central extension 
\[ 0\to \R\to \wh{\h}\to \h\to 0,\]
let $\E\subset \wh{\h}^*$ be the affine space of  linear functionals taking $1\in \R$ to $1$. The adjoint action of $H$ on $\wh{\h}$ 
defines an affine $H$-action on this space, with linear part the coadjoint action on $\h^*$. 
Conversely, suppose $\E$ is an affine space with an affine $H$-action, with underlying linear $H$-space the 
coadjoint representation. For $\mu\in \E$, the map $\h\to \h^*,\ \xi\mapsto \xi. \mu$, 
defined by the infinitesimal action, is a 
Lie algebra cocycle. Suppose this cocycle is skew-symmetric:  
\begin{equation}\label{eq:skew} \l\xi. \mu,\ \eta \r=-\l\eta. \mu,\ \xi\r\end{equation}
for all $\xi,\eta$. (This condition does not depend on the choice of $\mu$.) Then one obtains a central extension:
Take $\wh{\h}$ to be the vector space of affine-linear functions on $\E$, with bracket 
\[ [\wh{\xi},\wh{\eta}](\mu)=\l\xi. \mu,\ \eta \r.\]
Here $\xi,\eta$ are the linear functionals underlying $\wh{\xi},\wh{\eta}\in  \wh{\h}$. The affine space $\E$ is a Poisson submanifold of $\wh{\h}^*$, with symplectic leaves the coadjoint orbits. The action groupoid $H\times \E\rra \E$ has a canonical 
symplectic structure making it into a symplectic groupoid. 
See \cite[Example A.12]{al:coad} for further discussion.

\subsection{Affine moment maps}\label{app:A2}
Let $\E$ be an affine $H$-space as above, with underlying linear $H$-space the coadjoint representation.   
Given an  $H$-manifold $M$ with a closed invariant 2-form $\omega$, we may consider $\E$-valued affine moment maps 
\[ \Phi\colon M\to \E;\]
that is, $\Phi$ is $H$-equivariant and satisfies the moment map condition\footnote{Generating vector fields are defined in terms of their action on functions as $(\xi_M f)(m)=\f{d}{d t}|_{t=0}f\big(\exp(-t\xi).m\big)$.  With this convention, $\xi\mapsto \xi_M$ is a Lie algebra morphism.  Note that if $M=\E$ is an affine space, then 
	$\xi_M(\mu)=-\xi.\mu$.}
\[\omega(\xi_M,\cdot)=-\l \d\Phi,\xi\r.\] 
Note that the differential $\d\Phi$ is a 1-form with values in the linear space  $\h^*$, hence its pairing with $\xi$ is defined.
Examples of Hamiltonian $H$-spaces with $\E$-valued moment maps are the coadjoint orbits in $\E$, with the KKS symplectic structure.  

A necessary condition for the existence of an affine moment map (for some $\E$) is that the 1-forms $\alpha\in \Omega^1(M,\h^*)$, given as 
$\l\alpha,\xi\r=-\omega(\xi_M,\cdot)$, are exact. In fact, this condition is sufficient as well: 
\begin{proposition}\label{prop:E}
Let $M$ be a connected $H$-manifold, with an invariant closed 2-form $\omega\in \Omega^2(M)$. Suppose that the 1-form
\[ \alpha\in \Omega^1(M,\h^*)\]
given as $\l\alpha,\xi\r=\omega(\xi_M,\cdot)$ is exact. Let $\E$ be the affine space of all primitives of $\alpha$, and let  
\[ \Phi\colon M\to \E\]
be the map taking  any point of $M$ to the unique primitive vanishing at that point. 
Then $\E$ is an affine $H$-space, with underlying linear action the coadjoint action, satisfying the skew-symmetry \eqref{eq:skew}.  
Furthermore, $\Phi$ 
 is an affine moment map. 
\end{proposition}
\begin{proof}
Since $M$ is connected, a primitive of $\alpha$ is unique up to a constant
function with values in $\h^*$. This shows that $\E$ is an affine space over $\h^*$. The group $H$ acts on $\E
\subset C^\infty(M,\h^*)$ by 
$(h. f)(m)=\Ad_h \big(f(h^{-1}. m)\big)$; hence the difference of two elements transforms under the coadjoint action. 
For $f\in \E$ we have 
\[ \l\xi.f,\eta\r=\l \L_{\xi_M} f+\ad_\xi f,\eta\r=-\omega(\xi_M,\eta_M)
-\l f,[\xi,\eta]\r
,\] 
which is skew-symmetric in $\xi,\eta$. To verify that 
$\Phi$ (as in the proposition) is a moment map, fix $m_0\in M$.
Then $\Phi(m_0)(m)=\int_{m_0}^m \alpha=-\Phi(m)(m_0)$. Hence, the map  $m\mapsto \Phi(m)|_{m_0}$ is a primitive for $-\alpha$. This shows 
$\l\d\Phi,\xi\r=-\alpha(\xi)=-\omega(\xi_M,\cdot)$. 
\end{proof}

Note that $\E$, and hence the central extension of $\h$, is determined by the pullback of $\omega$ to any $G$-invariant submanifold of $M$. For example, if the action has a fixed point, or more generally if it admits an invariant isotropic submanifold, then the central extension is trivial, and the action admits an ordinary $\h^*$-valued moment map.


The constructions above apply to infinite-dimensional settings, provided that one has a reasonable notion of smooth dual. 
In particular, the affine action of $\Diff_+(\CC)$ on the space $\on{Hill}(\CC)$ of Hill operators (Section \ref{subsec:hill}) 
defines a central extension of $\Vect(\CC)$, the Virasoro algebra.

\section{Gauge theory constructions} \label{app:AB}
In this appendix, we review the Atiyah-Bott construction for principal $G$-bundles $P\to \Sigma$ over oriented surfaces with boundary. Here 
$G$ is any Lie group with an invariant inner product on its Lie algebra (we are mainly interested in the case $G=\PSL(2,\R)$).  

\subsection{Atiyah algebroid, connections}\label{subsec:atiyah}
Let $P\to M$ be a principal $G$-bundle. The groups of gauge transformations and automorphisms are denoted $\Gau(P)\subset \Aut(P)$. 
The \emph{Atiyah algebroid} $\on{At}(P)=TP/G\to M$ is the Lie algebroid whose sections are the 
$G$-invariant vector fields on $P$, that is, infinitesimal automorphisms. It fits into the exact sequence of Lie algebroids
\begin{equation}\label{eq:Atiyahsequence} 0\lra \g(P)\lra \At(P) \stackrel{\a}{\lra} TM\lra 0,\end{equation}
with the adjoint bundle $\g(P)=P\times_G \g$. On the level of sections, this is the exact sequence 
$0\to \gau(P)\to \aut(P)\to \on{Vect}(M)\to 0$. 
A principal connection $\theta\in \Omega^1(P,\g)^G$ is equivalent to a vector bundle splitting of the Atiyah sequence. 
This may be described by either of the bundle maps 
\[ s^\theta\colon \At(P)\to \g(P),\ \ \ j^\theta\colon TM\to \At(P)\]
called \emph{vertical projection} and \emph{horizontal lift}; thus $s^\theta|_{\g(P)}=\on{id}_{\g(P)}$, 
$\a\circ j^\theta=\on{id}_{TM},\ s^\theta\circ j^\theta=0$.  The section $s^\theta(v)$ corresponds to $\iota_v\theta$ under the identification  $\g(Q)\cong \Omega^0(Q,\g)^G$. 
We denote by 
\[ \d^\theta\colon \Omega^p(M,\g(P))\to  \Omega^{p+1}(M,\g(P))\]
the covariant derivative; in terms of the identification of 
forms $\beta\in\Omega^\bullet(M,\g(P))$ with $G$-basic forms 
$\wt{\beta}\in \Omega^\bullet(P,\g)$ we have $\wt{\d^\theta\beta}=\d\wt\beta+[\theta,\wt{\beta}]$. The curvature $F^\theta\in \Omega^2(M,\g(P))$ 
corresponds to the basic form $\wt{F}^\theta=\d\theta+\hh[\theta,\theta]$. 

%
The set $\A(P)$  of principal connections is an affine space over  
$\Omega^1(M,\g(P))$: given a smooth family of connections $\theta_t$ with $\theta_0=\theta$, the corresponding tangent vector 
$\beta\in \Omega^1(M,\g(P))$ is determined by either one of the equations
\[ \f{d}{d t}|_{t=0}\theta_t=\wt{\beta},\ \ \ 
\f{d}{d t}|_{t=0}s^{\theta_t}=\beta\circ \az,\ \ \   \f{d}{d t}|_{t=0}j^{\theta_t}=-\beta.\]

\begin{lemma} The generating vector fields for the $\Aut(P)$-action on $\A(P)$ are given by
	\begin{equation}\label{eq:genvf}
	 v_{\A(P)}|_\theta=-\d^\theta(s^\theta(v))-\iota_{\a(v)}F^\theta,\ \ v\in\mf{aut}(P)=\Gamma(\At(P)).\end{equation}	
	In particular, for 	 $v\in \Gamma(\g(P))=\Omega^0(M,\g(P))$, the generating vector field is $-\d^\theta v$. 
\end{lemma}
\begin{proof}
By our sign convention for generating vector fields, $v_{\A(P)}|_\theta=-v.\theta$. 	
Regarding $v$ as a $G$-invariant vector field $\wt{v}$ on $P$, we have 
\[ 	\wt{v.\theta}=\L_{\wt{v}}\theta=\d \iota_{\wt{v}}\theta+\iota_{\wt{v}}\d\theta=
(\d+[\theta,\cdot]) \iota_{\wt{v}}\theta+\iota_{\wt{v}} \wt{F}^\theta
=\wt{\d^\theta(s^\theta(v))}+\wt{\iota_{\a(v)}F^\theta}.\qedhere\]
\end{proof}

\subsection{Trivializations}
Suppose $P\to M$ admits a section $\iota\colon M\to P$, defining a trivialization
$P\cong M\times G$ such that $\iota(m)=(m,e)$ with the principal action $a. (m,g)=(m,ga^{-1})$. 
Connections $\theta\in \A(P)$ are described in terms of their connection 1-forms $A=\iota^*\theta$ by 
\[ \theta=\Ad_{g^{-1}}(A)+\pr_2^*\theta^L\]
The trivial connection (given by $\theta=\pr_2^*\theta^L$) gives a splitting of the Atiyah algebroid 
$\At(P)\cong TM\oplus \g(P)$. The trivialization of $P$ determines a trivialization of all its associated bundles, and in particular 
gives an isomorphism
\begin{equation}\label{eq:iso1} \g(P)\cong M\times \g.
\end{equation}
We have 
\[ s^\theta(X,\xi)=\xi+\iota_X A,\ \ j^\theta(X)=(X,-\iota_X A).\] 
The covariant derivative is the operator on $\Omega^\bullet(M,\g(P))\cong \Omega^\bullet(M,\g)$ given by $\d_A=\d+[A,\cdot]$; the curvature is $F_A=\iota^* \wt{F}^\theta=\d A+\hh [A,A]$. 

\begin{remark}[Signs, I]\label{rem:signs1}
We stress that the isomorphism 
\begin{equation}\label{eq:iso2}\gau(P)=C^\infty(M,\g)\end{equation} 
given by \eqref{eq:iso1} differs by sign from the `standard' identification as the Lie algebra of  
$\Gau(P)=C^\infty(M,G)$ (with pointwise multiplication). In fact, \eqref{eq:iso2} takes a function $\xi\in C^\infty(M,\g)$ to the vertical vector field whose restriction to $\{m\}\times G$ is $\xi(m)^R$. In particular, \eqref{eq:iso2} induces \emph{minus} the pointwise Lie bracket on $C^\infty(M,\g)$. More generally, 
$\mf{aut}(P)=C^\infty(M,\g)\rtimes \on{Vect}(M)$ with the bracket 
\[ [(\xi,X),(\eta,Y)]=(-[\xi,\eta],\L_X\eta-\L_Y\xi).\]
\end{remark}


\subsection{Central extensions}\label{app:A3}
Suppose  $Q\to \CC$ is a principal $G$-bundle over a compact oriented 1-manifold, and that the Lie algebra $\g$ carries an 
invariant metric (denoted by a dot).  Then $\g(Q)$ inherits a bundle metric. The affine space $\A(Q)$
of connections has $\gau(Q)^*=\Omega^1(\CC,\g(Q))$ as its underlying linear space, 
where the pairing with $\gau(Q)=\Omega^0(\CC,\g(Q))$ is given by the metric followed by integration. This identification takes 
the linear part of the gauge action to the coadjoint action. Furthermore, $\l \xi.\theta,\eta\r=-\l \d^\theta\xi,\eta\r$ 
satisfies the skew-symmetry condition 
\eqref{eq:skew}. By the discussion of Appendix \ref{app:A1}, this defines a 
central extension of $\gau(Q)$.  

\begin{remark}[Signs, II]\label{rem:signs2} For trivial bundles $Q=\CC\times G$, one often uses the 
`standard' identification $\gau(Q)\cong C^\infty(\CC,\g)$ to define the pairing. As explained in Remark \ref{rem:signs1}, this is opposite to the identification coming from $\g(Q)\cong \CC\times \g$, and hence results in the opposite pairing. 
\end{remark}

We may also consider the larger group $\Aut(Q)$ of all principal bundle automorphisms, with Lie algebra 
$\aut(Q)=\Gamma(\At(Q))$; its smooth dual is  $\aut(Q)^*=\Omega^1(\CC,\At(Q)^*)$. 
Let\begin{equation}\label{eq:EQ} \E(Q)\subset \Gamma(\on{Sym}^2 \At(Q)^*)\end{equation}be the affine space of all fiberwise quadratic forms on $\At(Q)$ whose restriction to $\g(Q)$ is given by $\xi\mapsto \hh \xi\cdot\xi$.  
The underlying linear space consists of quadratic forms on $\At(Q)$ whose restriction to $\g(Q)$ is zero; it is identified with 
the space $\Omega^1(\CC,\At(Q)^*)$, where an element $\gamma$ of this space defines the quadratic form $w\mapsto \l \gamma(\a(w)),w\r$. 
The group $\Aut(Q)$ acts on $\E(Q)$, and the underlying linear action is 
coadjoint action. One may also check that it satisfies the skew-symmetry property \eqref{eq:skew}; hence one obtains a 
central extension of $\aut(Q)$. 

\begin{remark}
	There is a natural map \[ \E(Q)\to \A(Q),\ \ \phi\mapsto \theta\] 
given by $s^\theta(v)\cdot\xi=-2\phi(v,\xi)$ for $v\in \Gamma(\At(Q),\ \xi\in \g(Q)$
(where we think of $\phi$ as a symmetric bilinear form).  
This map is affine with respect to the quotient map $\Omega^1(\CC,\At(Q)^*)\to \Omega^1(\CC,\g(Q)^*)$, 
and determines a Lie algebra morphism $\wh{\gau}(Q)\to \wh{\aut}(Q)$ lifting the natural inclusion. 
\end{remark}
\begin{remark}
There is a natural 
$\Aut(Q)$-equivariant section
section \begin{equation}\label{eq:section} \A(Q)\to \E(Q),\ \ \theta\mapsto \hh s^\theta\stackrel{.}{\vee} s^\theta.\end{equation} 
Here $\vee$ denotes the product in the symmetric algebra. The map  \eqref{eq:section} is the moment map for the $\Aut(Q)$-action on $\A(Q)$, in the sense that it 
restricts to moment maps  on the coadjoint orbits $\O\subset \A(Q)$.

For a trivial bundle $Q=\CC\times G$, we have $\A(Q)=\Omega^1(\CC,\g)$ and $\E(Q)=\Omega^1(\CC,\g)\times |\Omega|^2_\CC$; 
in terms of these identifications the map is 
\[ A\mapsto (A,\hh A\cdot A).\] 
It is gauge equivariant for the action $h.(A,q)=(h\bullet A,\,q+A\cdot h^*\theta^L-\hh h^*\theta^L\cdot h^*\theta^L$). 

\end{remark}

 \subsection{Atiyah-Bott}
We now assume that $\Sigma$ is a compact oriented surface, possibly with boundary, and that the Lie algebra 
$\g=\on{Lie}(G)$ comes with an invariant metric. Let $P\to \Sigma$ be a principal $G$-bundle.

\begin{definition}
The \emph{Atiyah-Bott form} on $\A(P)$ is given by 
\[ \omega_{AB}(a,b)=\int_{\Sigma} a\cdot b,\ \ \ \ a,b\in \Omega^1(\Sigma,\g(P)).\]
\end{definition}
The Atiyah-Bott form  $\omega_{AB}$  is closed, by translation invariance, and nondegenerate in the weak sense that $\omega_{AB}(a,\cdot)=0 \Leftrightarrow a=0$. The 2-form
is invariant under the action of the group $\bb\!\!\Aut_+(P)$ of principal bundle automorphisms whose base map lies in $\bb\Diff_+(\Sigma)$. We are interested in a moment map for this action. The Lie algebra  $\bb\!\!\mf{aut}(P)$ consists of 
$G$-invariant vector fields on $P$ that are tangent to $\p P$; these are the sections of a Lie algebroid 
$\bb\!\!\At(P)$.
Let $\p P=P|_{\p\Sigma}$ the boundary restriction of 
$P$. For a connection $\theta$, denote by $\p\theta$ its pullback to a connection on $P$. 

\begin{proposition}
	The action of $\bb\!\!\Aut_+(P)$ on $\A(P)$ has moment map 
\[ \A(P)\to \Omega^2(\Sigma,\At(P)^*)\times \E(\p P),\ \theta\mapsto (-F^\theta\cdot s^\theta,\hh s^{\p\theta}\stackrel{.}{\vee} s^{\p\theta}).\]	
For the action of $\Gau(P)$, the moment map takes values in $\Omega^2(\Sigma,\g(P))\times \A(\p P)$, and is given by 
$(-F^\theta,\p\theta)$. 
\end{proposition}
\begin{proof}
	Let $v\in \bb\aut(P)=\Gamma(\bb\!\!\At(P))$, with boundary restriction $\p v$. 
	Given  $b\in T_\theta\A(P)=\Omega^1(\Sigma,\g(P))$, we have, using \eqref{eq:genvf},  
	\begin{align*} \omega_{AB}(v^\sharp|_\theta,\, b)&=
	\int_\Sigma v^\sharp|_\theta\cdot b\\
	&=-\int_\Sigma \iota_{\a(v)}F^\theta\cdot b-\int_\Sigma \d^\theta(s^\theta(v))\cdot b\\
	&=\int_\Sigma \iota_{\a(v)} b\cdot 
		F^\theta +
	\int_\Sigma s^\theta(v)\cdot \d^\theta b-\int_{\pS} s^{\p\theta}(\p v)\cdot \p b.
	\end{align*}
	where  $\p b\in T_{\p \theta}\A(\p P)$ is the pullback of $b$ to the boundary. 
	Suppose $b\in T_\theta\A(P)$ is realized as the velocity vector for a family of connections $\theta_t$ with 
	\[ \theta_0=\theta,\ \ \ \f{d}{d t}\Big|_{t=0}\theta_t=\wt{b}.\]
	Then $\f{d}{d t}\big|_{t=0} s^{\theta_t}=b\circ \az,\  \f{d}{d t}\big|_{t=0} F^{\theta_t}=\d^{\theta_t}b$, 
	and therefore 
	\[ \f{d}{d t}\Big|_{t=0}  F^{\theta_t}\cdot s^{\theta_t}
	= (b\circ \az)\cdot F^\theta+s^\theta\cdot \d^\theta b
	.\]
Here $b\circ \az,\ s^\theta$ are regarded as sections of $\At(P)^*\otimes \g(P)$, while $F^\theta,\ \d^\theta b$ are elements of 
$\Omega^2(\Sigma,\g(P))$. 
	For the boundary term, we note 	
	\[ \f{d}{d t}\Big|_{t=0}  \left(\hh s^{\p\theta_t}\stackrel{.}{\vee} s^{\p\theta_t}\right) 
	=s^{\p\theta}\stackrel{.}{\vee}(\p b\circ \az)
	=s^{\p\theta}\cdot\p b\]
	where the last equality comes from the inclusion $\Omega^1(\pS,\At(\p P)^*)\hra \Gamma(\on{Sym}^2(\At(\p P)^*))$. 
	This gives 
	\[  \omega_{AB}(v^\sharp|_\theta,\, b)=
	\f{d}{d t}\Big|_{t=0} \int_\Sigma F^{\theta_t}\cdot s^{\theta_t}(v)-\int_{\pS} \Big(  \f{d}{d t}\Big|_{t=0} \hh s^{\p\theta_t}\vee s^{\p \theta_t}\Big) (\p v,\cdot)
	\Big).\qedhere \]
\end{proof}

\begin{remark}[Signs, III]\label{rem:signs3}
The sign in the moment map 	for $\Gau(P)$ depends on the identifications $\gau(P)\cong \Omega^2(P,\g(P))$ and 
$\gau(\p P)\cong \Omega^1(\p P,\g(\p P))$. For trivial bundles, one often uses the opposite pairing, resulting in a sign change of the moment map. See Remarks \ref{rem:signs1}, \ref{rem:signs2}.
\end{remark}

\section{The trumpet moduli space as a symplectic cross-section}\label{app:groupoid}
Our starting point will be the interpretation of $N$ as a symplectic slice in a symplectic groupoid 
$\ca{G}\rra \on{Hill}(S^1)$, see \cite{al:coad}.  As a groupoid,  $\G\cong \on{Hill}(S^1)\times \Diff_\oz(S^1)$
(an action groupoid), with source and target maps $\sz(T,\FF)=T,\ \tz(T,\FF)=\FF. T$ and groupoid 
multiplication $(T_1,\FF_1)\circ (T_2,\FF_2)=(T_2,\FF_1\circ \FF_2)$. 
In terms of this `left trivialization', the symplectic structure $\omega_\G$ is given by 
\cite[Equation 47]{al:coad}
\begin{equation}\label{eq:lefttrivialization}
\omega_\G=\int_{S^1} \Big( \d T\wedge \f{\d \mathsf{F}}{\mathsf{F}'}+T\, \f{\d \mathsf{F}}{\mathsf{F}'}\wedge 
\big(\f{\d \mathsf{F}}{\mathsf{F}'}\big)'-\f{1}{4} \big(\f{\d \mathsf{F}}{\mathsf{F}'}\big)'''\wedge \big(\f{\d \mathsf{F}}{\mathsf{F}'}\big)\Big).\end{equation}
%
 

For our description of the 2-form for the trumpet, it will be more convenient to work with `right trivialization', expressing the 2-form in terms of $(T_0,\FF)$ where 
\begin{equation}\label{eq:Tform} T=\FF^{-1}. T_0=(\FF')^2 T_0+\hh \S(\FF).\end{equation}

\begin{proposition}
	In right trivialization, 
	\[  \omega_\G=\d \int_{S^1}\ 	\ \Big( T_0 \FF' -\f{1}{4}(\FF')^{-1}\d \FF''\Big)\]
\end{proposition}
\begin{proof}
	This is based on a straightforward but lengthy calculation, substituting \eqref{eq:Tform} and simplifying. Here are some relevant steps. 
	Using the formula for the exterior differential of the Schwarzian derivative 
	(\cite[Lemma A.2]{al:bos}) one shows 
	\[ \int_{S^1}\  \d\S(\FF)\wedge \big(\f{\d \mathsf{F}}{\mathsf{F}'}\big)=
	\int_{S^1}\  \f{\d \FF'\wedge \d\FF''}{(\mathsf{F}')^2}
	\]
	This then implies
	\[ \int_{S^1}\ \d T\wedge \f{\d \mathsf{F}}{\mathsf{F}'}
	=\int_{S^1}\  \Big(\FF' \d T_0\wedge \d \mathsf{F}+2T_0\,\d \FF'\wedge \d \FF
	+\f{1}{2} \f{\d \FF'\wedge \d\FF''}{(\mathsf{F}')^2}\Big)\]
	Furthermore, 
	\[ 	\int_{S^1}\ T\, \f{\d \mathsf{F}}{\mathsf{F}'}\wedge 
	\big(\f{\d \mathsf{F}}{\mathsf{F}'}\big)'=\int_{S^1}\  \Big(T_0\,\d \FF'\wedge \d \FF 
	+\f{1}{2} \S(\FF)\f{\d \FF\wedge \d\FF'}{(\mathsf{F}')^2}\Big)\]
	and 
	\[-\f{1}{4} \int_{S^1}\  \big(\f{\d \mathsf{F}}{\mathsf{F}'}\big)'''\wedge \big(\f{\d \mathsf{F}}{\mathsf{F}'}\big)
	= \int_{S^1}\  \Big(-\f{1}{4} \f{\d \FF'\wedge \d\FF''}{(\mathsf{F}')^2}
	-\f{1}{2} \S(\FF)\f{\d \FF\wedge \d\FF'}{(\mathsf{F}')^2}
	\Big).
	\]
	Adding these three contributions, the formula for $\omega_\G$ follows. 
\end{proof}

Passage to the slice $N=\R_{>0}\times \Diff_\oz(S^1)$ amounts to putting $T_0(x)=-\f{\ell^2}{4}$, resulting in the formula \eqref{eq:omegan}.

\end{appendix}

\bibliographystyle{amsplain} 

\def\cprime{$'$} \def\polhk#1{\setbox0=\hbox{#1}{\ooalign{\hidewidth
			\lower1.5ex\hbox{`}\hidewidth\crcr\unhbox0}}} \def\cprime{$'$}
\def\cprime{$'$} \def\cprime{$'$} \def\cprime{$'$} \def\cprime{$'$}
\def\polhk#1{\setbox0=\hbox{#1}{\ooalign{\hidewidth
			\lower1.5ex\hbox{`}\hidewidth\crcr\unhbox0}}} \def\cprime{$'$}
\def\cprime{$'$} \def\cprime{$'$} \def\cprime{$'$} \def\cprime{$'$}
\providecommand{\bysame}{\leavevmode\hbox to3em{\hrulefill}\thinspace}
\providecommand{\MR}{\relax\ifhmode\unskip\space\fi MR }
\providecommand{\MRhref}[2]{%
	\href{http://www.ams.org/mathscinet-getitem?mr=#1}{#2}
}
\providecommand{\href}[2]{#2}

\end{document}